\documentclass[10pt,leqno]{amsart}
\usepackage[letterpaper,margin=25mm,headheight=12pt,headsep=16pt,footskip=24pt]{geometry}
\usepackage{amsmath,amssymb,mathtools,mathrsfs,booktabs,array}
\usepackage{microtype}
\usepackage[colorlinks=true,linkcolor=blue,citecolor=blue,urlcolor=blue]{hyperref}
\newtheorem{theorem}{Theorem}[section]
\newtheorem{proposition}[theorem]{Proposition}
\newtheorem{lemma}[theorem]{Lemma}
\newtheorem{corollary}[theorem]{Corollary}
\newtheorem{conjecture}[theorem]{Conjecture}
\theoremstyle{definition}
\newtheorem{definition}[theorem]{Definition}
\newtheorem{remark}[theorem]{Remark}
\newcommand{\C}{\mathbb C}
\newcommand{\PP}{\mathbb P}
\newcommand{\Z}{\mathbb Z}
\newcommand{\OO}{\mathcal O}
\renewcommand{\AA}{\mathcal A}
\newcommand{\Db}{\mathrm D^{\mathrm b}}
\newcommand{\HS}{\operatorname{HS}}
\newcommand{\HH}{\operatorname{HH}}
\newcommand{\Hom}{\operatorname{Hom}}
\newcommand{\RHom}{\operatorname{RHom}}
\newcommand{\Ext}{\operatorname{Ext}}
\newcommand{\End}{\operatorname{End}}
\newcommand{\cEnd}{\mathcal E nd}
\newcommand{\Ker}{\operatorname{Ker}}
\newcommand{\Image}{\operatorname{Image}}
\newcommand{\Jac}{\operatorname{Jac}}
\newcommand{\Sym}{\operatorname{Sym}}
\newcommand{\Pic}{\operatorname{Pic}}
\newcommand{\Perf}{\operatorname{Perf}}
\newcommand{\coh}{\operatorname{coh}}
\newcommand{\id}{\operatorname{id}}
\newcommand{\rk}{\operatorname{rk}}
\newcommand{\Cl}{\operatorname{Cl}}
\newcommand{\Fun}{\operatorname{Fun}}
\newcommand{\Bl}{\operatorname{Bl}}
\newcommand{\PGL}{\operatorname{PGL}}
\newcommand{\Spec}{\operatorname{Spec}}
\newcommand{\Kf}{\mathcal K_f}

\newcommand{\rootstack}[2]{\sqrt[2]{(#1,#2)}}
\newcommand{\alg}{_{\mathrm{alg}}}
\newcommand{\LL}{\mathbf L}
\newcommand{\RR}{\mathbf R}
\newcommand{\Ktop}{\mathrm K^{\mathrm{top}}}
\newcommand{\modu}{\mathrm{mod}}
\DeclareMathOperator{\ch}{ch}
\DeclareMathOperator{\coker}{coker}
\DeclareMathOperator{\Rep}{Rep}

\title[Categorical reconstruction of del Pezzo surfaces]{Categorical reconstruction of del Pezzo surfaces: Hochschild--Serre algebras and spinor modifications}
\author{Xun Lin}
\address{School of Science and Engineering, The Chinese University of Hong Kong, Shenzhen, Shenzhen, Guangdong 518172, China}
\email{lin-x18@tsinghua.org.cn}
\author{Marco Rampazzo}
\address{Yau Mathematical Sciences Center, Tsinghua University, Shuangqing Complex Bldg., Haidian district, Beijing, China}
\email{marcorampazzo@mail.tsinghua.edu.cn, marco.rampazzo.90@icloud.com}
\author{Shizhuo Zhang}
\address{School of Mathematics, Sun Yat-sen University, Guangzhou 510275, P. R. China}
\email{zhangshzh28@mail.sysu.edu.cn}
\hypersetup{pdftitle={Categorical reconstruction of del Pezzo surfaces: Hochschild--Serre algebras and spinor modifications},pdfauthor={Xun Lin, Marco Rampazzo, and Shizhuo Zhang}}
\date{}
\begin{document}
\begin{abstract}
We prove that, for every smooth complex del Pezzo surface of degree at most four, the enhanced right orthogonal to the structure sheaf determines the surface up to isomorphism. In degrees one, two, and three, we recover the anticanonical equation from intrinsic pieces of the Hochschild–Serre algebra via graded matrix factorizations; in degree four, the relevant Serre diagonal recovers the orbifold canonical ring associated with the pencil of quadrics. We also give an alternative proof in degrees one and two using the Bertini and Geiser involutions, equivariant topological K-theory, and classical Torelli.
Then, we study the Clifford component associated with a conic bundle structure over the projective line. In every degree at most four, we construct an abstract spinor bundle whose spinor modification produces another, generally non-isomorphic, del Pezzo surface. We show via Kuznetsov’s modification theorem that the relevant base-linear equivalences are precisely those induced by spinor modifications.
\end{abstract}
\maketitle

\section{Introduction}
\subsection{Reconstruction after removing the structure sheaf}
Derived reconstruction is usually formulated for the whole category $\Db(X)$. For a smooth projective variety whose canonical or anticanonical class is ample, the theorem of Bondal and Orlov reconstructs the variety from this category \cite{BO01}. The purpose of the present paper is to ask how much of that conclusion survives after one discards a single, geometrically distinguished object.

The question has a canonical formulation for a smooth complex Fano variety $X$. Kodaira vanishing gives $H^i(X,\OO_X)=0$ for $i>0$, while $H^0(X,\OO_X)=\C$; hence $\OO_X$ is exceptional. We may therefore form the admissible subcategory
\begin{equation}\label{eq:CX}
\mathcal C_X:=\langle\OO_X\rangle^\perp,\qquad
\Db(X)=\langle\mathcal C_X,\OO_X\rangle.
\end{equation}
Unlike a Kuznetsov component defined by deleting a longer exceptional collection, $\mathcal C_X$ involves no auxiliary choice of such a collection. At the same time, it no longer records the position of $\OO_X$ or the gluing that places the two terms of \eqref{eq:CX} inside $\Db(X)$. Thus an equivalence of the orthogonal categories does not formally extend to an equivalence of the ambient derived categories. The conjecture attributed to Alexander Perry predicts that the missing exceptional object nevertheless carries no indispensable isomorphism-class information.
\begin{conjecture}\label{conj:perry}
For smooth complex Fano varieties $X$ and $X'$, a $\C$-linear enhanced equivalence $\mathcal C_X\simeq\mathcal C_{X'}$ implies $X\cong X'$.
\end{conjecture}
The word ``enhanced'' is essential here. Indeed, the enhancement hypothesis ensures that an equivalence preserves the Hochschild--Serre algebra used in our reconstruction: its definition involves derived natural transformations and their composition, which are transported by dg Morita equivalences \cite{Toe07}. Such compatibility is not automatic for triangulated equivalences in general. In fact, a triangulated category may admit non-equivalent dg enhancements \cite{RVdB19}. Therefore, we work with the standard smooth proper dg enhancements and $\C$-linear Morita equivalences between them. The geometric functors are represented by Fourier--Mukai kernels. Consequently, the Serre functor and the derived spaces of natural transformations between its powers are transported by an equivalence.

The restriction to degree at most four reflects the geometry of the moduli problem. By the classical classification \cite[Section 3.4]{DI09}, smooth complex del Pezzo surfaces of degrees $5,6,7$, and $9$ are unique up to isomorphism in each degree, while degree $8$ has just the two classes $\PP^1\times\PP^1$ and $\Bl_p\PP^2$. For $1\leq d\leq4$, there are positive-dimensional families of pairwise non-isomorphic surfaces. Thus, in this range, categorical reconstruction must recover the moduli of the surface beyond its numerical degree.

Our main result verifies the conjecture within the class of del Pezzo surfaces of degree at most four. We abbreviate the orthogonal category by
\begin{equation}\label{eq:AS}
\AA_S:=\mathcal C_S=\langle\OO_S\rangle^\perp.
\end{equation}
In degree four, the reconstruction theorem was already proved by Elagin \cite[Theorem 3.1]{Ela26}. Our contribution in this case is a proof through the Hochschild--Serre algebra: an intrinsic graded subalgebra recovers the orbifold canonical ring of the associated weighted projective line, and hence the five-point discriminant of the pencil of quadrics.

The category $\AA_S$ is canonical and, as we prove, determines $S$. After choosing a conic-bundle morphism on $S$, one obtains a smaller Clifford component. That component is unchanged by spinor modification and can therefore be shared by non-isomorphic del Pezzo surfaces. Comparing these two categories, and the distinguished spinor that relates them, organizes the results of the paper.

\subsection{The Hochschild--Serre algebra and intrinsic symmetries}
For a smooth proper dg category $\AA$ with Serre functor $S_{\AA}$, set
\begin{equation}\label{eq:HSintro}
\HS^{p,q}(\AA):=\Hom_{\Fun(\AA,\AA)}(\id_{\AA},S_{\AA}^q[p]),
\qquad \HS(\AA):=\bigoplus_{p,q\in\Z}\HS^{p,q}(\AA).
\end{equation}
Here the Hom space is taken in the homotopy category of derived dg quasi-functors. Composition twisted by powers of the Serre functor makes this a bigraded algebra; its Morita invariance is established in \cite[Definition 3.3 and Theorem 3.4]{LZ24}. The relevant pieces of $\HS(\AA_S)$ recover a Jacobian algebra in degrees one, two, and three, and an orbifold canonical ring in degree four.

The Serre functor provides a second intrinsic mechanism in the two double-cover cases. In degree two, let $\gamma$ be the Geiser involution of the anticanonical double plane; in degree one, let $\beta$ be the Bertini involution of the bi-anticanonical double cover. The rotation calculation in Proposition~\ref{prop:involution} gives
\begin{equation}\label{eq:involintro}
S_{\AA_S}^2\cong\gamma^*[3]\quad(d=2),\qquad
S_{\AA_S}^3\cong\beta^*[5]\quad(d=1).
\end{equation}
Thus the deck transformations are recovered intrinsically from $\AA_S$. We verify the coherence needed to pass an equivalence to equivariant categories. The cyclic-cover decomposition then isolates the odd topological K-theory of the branch curve; its Hodge structure and Euler pairing give the principally polarized Jacobian. Classical Torelli recovers the curve. Its canonical model extends the curve isomorphism to the plane in degree two and to the quadric-cone model of the weighted plane in degree one. Rescaling the covering coordinate lifts the ambient isomorphism to the double covers. The complete argument is given in Theorem~\ref{thm:equiv}.

\subsection{Main results}
Our first theorem is the categorical Torelli statement. It is proved in Section~\ref{subsec:proofTorelli}; Theorem~\ref{thm:equiv} gives a second proof in degrees one and two.
\begin{theorem}\label{thm:main}
Let $S$ and $S'$ be smooth complex del Pezzo surfaces with $K_S^2,K_{S'}^2\leq4$. If
\[
\Phi:\AA_S\xrightarrow{\sim}\AA_{S'}
\]
is a $\C$-linear equivalence of enhanced categories, then $K_S^2=K_{S'}^2$ and $S\cong S'$.
\end{theorem}
Additivity of Hochschild homology recovers the degree from
\[d=11-\dim\HH_0(\AA_S).\]
Once $d=K_S^2$ is fixed, the reconstruction algebra is summarized by the following table. Generated subalgebras are understood to contain the unit.
\begin{center}\small
\begin{tabular}{@{}cp{.33\textwidth}p{.51\textwidth}@{}}
\toprule
$d$ & intrinsic algebra in $\HS(\AA_S)$ & reconstructed datum\\
\midrule
1 & $\langle\HS^{2,-1},\HS^{4,-2}\rangle\alg$ & weighted Jacobian algebra of the sextic in $\PP(1,1,2,3)$\\[3pt]
2 & $\langle\HS^{2,-1}\rangle\alg$ & Jacobian algebra of the branch quartic\\[3pt]
3 & $\langle\HS^{2,-1}\rangle\alg$ & Jacobian algebra of the cubic equation\\[3pt]
4 & $\bigoplus_{p\geq0}\HS^{-p,p}$ & orbifold canonical ring and the five-point discriminant\\
\bottomrule
\end{tabular}
\end{center}
The second theorem concerns the smaller category. Choose a conic bundle $f:S\to B=\PP^1$. The quadric-fibration decomposition and Beilinson's collection on $B$ give
\begin{equation}\label{eq:sodintro}
\Db(S)=\langle\Kf,E_f,\OO_S\rangle,\quad
\AA_S=\langle\Kf,E_f\rangle,\quad E_f=f^*\OO_B(-1),
\end{equation}
where $\Kf=\Ker(Rf_*)$ is the Clifford component.
The construction and proof are given in Theorem~\ref{thm:partners}.
\begin{theorem}\label{thm:spinorintro}
Fix $1\leq d\leq4$, put $r=8-d$, and let $D=t_1+\cdots+t_r\subset\PP^1$ be reduced. There is a nonempty open set of direction data $u=(u_1,\ldots,u_r)\in(\PP^1)^r$ such that
\[
S_u=\Bl_{(t_1,u_1),\ldots,(t_r,u_r)}(\PP^1\times\PP^1)
\]
is a del Pezzo surface of degree $d$, and the first projection is a conic bundle $f_u:S_u\to\PP^1$ with discriminant $D$. For any two such data $u,v$, there is an explicit rank-two abstract spinor bundle $\mathcal F_v$ on $S_u$ such that
\[
(S_u/\PP^1)_{\mathcal F_v}\cong S_v/\PP^1,\qquad
\mathcal K_{f_u}\simeq\mathcal K_{f_v}.
\]
The equivalence is $\PP^1$-linear and t-exact. For general $u,v$, the surfaces are non-isomorphic. Modulo automorphisms of the second factor, the fixed-discriminant family has dimension $r-3=5-d$.
\end{theorem}
Theorem~\ref{thm:main} shows that $\AA_S$ determines the isomorphism class of a del Pezzo surface of degree $d\leq4$. Theorem~\ref{thm:spinorintro} constructs explicit pairs of non-isomorphic del Pezzo surfaces with equivalent Clifford components, showing that $\Kf$ alone does not determine the surface.

The notation $(S_u/\PP^1)_{\mathcal F_v}$ denotes the spinor modification recalled in Section~\ref{subsec:clifford}. By \cite[Theorem 1.2]{Kuz25}, the even Clifford algebra of this modification is $f_{u*}\cEnd(\mathcal F_v)$. We identify this algebra with the normalized even Clifford algebra of $S_v/\PP^1$. Thus the theorem identifies the actual partner surface, in addition to constructing an equivalence of components.

\begin{remark}
Additivity of Hochschild homology also gives
\[
\dim_{\C}\HH_0(\Kf)=10-d,\qquad d=K_S^2.
\]
Thus the enhanced Clifford component alone determines the degree of the del Pezzo surface. As an admissible subcategory of $\AA_S$, the category $\Kf$ inherits a dg enhancement by taking the corresponding full dg subcategory of the enhancement of $\AA_S$ \cite[Remark 5.4]{BLS16}.
\end{remark}

Over the projective line, the fixed-discriminant description extends to the following characterization of spinor modification, proved in Theorem~\ref{thm:hyperbolic}.
\begin{theorem}\label{thm:hyperbolicintro}
Let $k$ be an algebraically closed field of characteristic different from two, and let $f_i:X_i\to B=\PP^1_k$, $i=1,2$, be flat conic bundles with smooth generic fibers and reduced discriminant divisors $\Delta_i$. The following conditions are equivalent:
\begin{enumerate}
\item $\Delta_1=\Delta_2$ as divisors on the fixed base $B$;
\item $X_1/B$ and $X_2/B$ are hyperbolically equivalent;
\item $X_1/B$ and $X_2/B$ are spinor modifications of one another.
\end{enumerate}
In particular, Kuznetsov's conjecture \cite[Conjecture 1.4]{Kuz25} holds for simply degenerating conic bundles over $\PP^1_k$ in this setting.
\end{theorem}
 For a quartic del Pezzo surface, Proposition~\ref{prop:ten} identifies its ten conic-bundle classes with the two exceptional simple objects over each of the five stacky points associated with the anticanonical pencil. Taking the right orthogonal to one such object removes the corresponding stacky point and produces a four-point model of the selected Clifford component.

\subsection{Outline of the proofs}
For $d=1,2,3$, the anticanonical models are the quasi-smooth hypersurfaces
\[
S_1\subset\PP(1,1,2,3),\quad S_2\subset\PP(1,1,1,2),\quad S_3\subset\PP^3
\]
of weighted degrees $6,4,3$, respectively. Orlov's theorem identifies $\AA_S$ with the corresponding category of graded matrix factorizations. The kernel formula of Ballard--Favero--Katzarkov decomposes every Hochschild--Serre component into an identity sector and possible twisted sectors. In the generating bidegrees displayed above the relevant pieces are untwisted. On the diagonal $\HS^{2e,-e}$ their products are ordinary products in the Jacobian ring. They generate the full Jacobian algebra: one bidegree suffices for $d=2,3$, while the weight-two variable for $d=1$ forces the second bidegree $\HS^{4,-2}$. Graded homogeneous reconstruction then recovers the defining equation.

For $d=4$, the surface is an intersection of two quadrics in $\PP^4$. The pencil has a reduced discriminant divisor $D\subset\PP^1$ of degree five, and the residual category is the derived category of the stacky line $\mathcal X_D=\rootstack{\PP^1}{D}$. Since its Serre functor is $(-\otimes\omega_{\mathcal X_D})[1]$, the diagonal $\bigoplus_{p\geq0}\HS^{-p,p}$ is the orbifold canonical ring. It has a presentation
\[
\frac{\C[u,v,w]}{(w^2-\ell_1(u,v)\cdots\ell_5(u,v))},\qquad
\deg u=\deg v=2,\quad\deg w=5,
\]
which recovers the unordered five-point divisor and hence the simultaneous diagonalization class of the pencil.

The equivariant proof uses \eqref{eq:involintro}. Every equivalence intertwines the intrinsic involutions, and a scalar adjustment makes the intertwiner compatible with their square relations. The resulting equivalence of equivariant categories preserves topological K-theory with its Hodge structure and Euler pairing. The cyclic-cover decomposition exhibits the branch category plus one exceptional object in degree two and plus two exceptional objects in degree one. Their odd topological K-groups vanish, so the resulting polarized weight-one Hodge structure is $H^1(C,\Z)$. Classical Torelli recovers the branch curve. Its canonical embedding extends the isomorphism to $\PP^2$, respectively to the quadric-cone model of $\PP(1,1,2)$, and the double covers are then recovered.

For the spinor statement, fix the discriminant points and vary the second coordinates of the blow-up centers on $\PP^1\times\PP^1$. An elementary transformation constructs $\mathcal F_v$ and verifies
\[
Rf_{u*}\mathcal F_v=0,\qquad \det\mathcal F_v\cong\omega_{S_u/\PP^1}.
\]
Applying the same construction on $S_v$ identifies its diagonal member with the normalized canonical spinor. Computing the two relative endomorphism algebras proves the claimed modification. If $i:\Kf\hookrightarrow\AA_S$, relative duality identifies the adjoint image of the exceptional object as
\[
i^!E_f\cong\mathcal F^0_{S/B}\otimes f^*\OO_B(-1).
\]
Transporting the analogous object from a partner gives its modifying spinor, with the displayed base twist.

For hyperbolic equivalence, write each conic bundle as a ternary line-bundle-valued quadratic form. Equal reduced discriminants give isometric shifted quadratic cokernel sheaves after a common twist normalization. The remaining Witt condition in \cite[Theorem 1.3]{Kuz24} is automatic over an algebraically closed field when the base is $\PP^1$. The converse follows from \cite[Proposition 1.1]{Kuz24}, while \cite[Corollary 1.3]{Kuz25} relates hyperbolic equivalence to spinor modification. In degree four, the five singular quadrics in the anticanonical pencil index five pairs of conic-bundle classes. The canonical heart in \cite[Section 2.4]{Ela26} identifies these classes with the ten exceptional simple objects; perpendicular reduction gives the associated four-point categories.

\subsection{Comparison with categorical Torelli for Kuznetsov components}
The distinction between $\AA_S$ and $\Kf$ is parallel to a familiar phenomenon for Gushel--Mukai varieties. In a decomposition
\[
\Db(X)=\langle\mathcal{K}u(X),E_1,\ldots,E_s,\OO_X\rangle,
\]
the Kuznetsov component is the residual category obtained by removing the standard exceptional block, while $\langle\OO_X\rangle^\perp$ retains the other exceptional objects. The smaller component may retain only a birational or period-partner class; additional gluing data can be encoded by a distinguished object obtained from the exceptional block.

Indeed, two smooth complex GM threefolds have equivalent enhanced Kuznetsov components if and only if they are period partners or duals; in particular, such an equivalence implies birationality \cite[Theorem 1.9]{BP23}. For general ordinary GM threefolds, an equivalence preserving the distinguished object $j_X^!\mathcal U_X$, where $j_X:\mathcal{K}u(X)\hookrightarrow\Db(X)$ is the inclusion, determines $X$ up to isomorphism \cite[Theorem 9.2]{JLLZ24}. Thus the additional datum is precisely the adjoint image of the exceptional object retained by the orthogonal to $\OO_X$. In our setting its analogue is
\[
i^!E_f\cong\mathcal F^0_{S/B}\otimes f^*\OO_B(-1),
\]
as established in Proposition~\ref{prop:adjoint}.

For a del Pezzo conic bundle, $\Kf$ is invariant under spinor modification, whereas the larger category $\AA_S$ reconstructs the surface. The corresponding distinguished object is explicit. If $S'$ is a spinor partner and $\Phi_{S',S}:\mathcal K_{f'}\to\Kf$ is the normalized induced equivalence, then
\[
\Phi_{S',S}(i_{S'}^!E_{f'})\cong\mathcal F_{S'/S}\otimes f^*\OO_B(-1).
\]
Thus the partner's adjoint image becomes the abstract spinor selecting that partner from the Clifford category. This provides a surface example of reconstruction using a distinguished object.

\subsection{Upcoming work}
The results above form part of a broader program to study the geometric information carried by the Hochschild--Serre algebra of a nontrivial admissible subcategory. For an admissible subcategory $\AA\subset\Db(X)$, the bigraded algebra $\HS(\AA)$ records natural transformations from the identity functor to shifted powers of the Serre functor, together with their compositions. The objective is to determine when this intrinsic multiplicative structure recovers geometric data after passage to the subcategory. The two principal targets are reconstruction of $X$ from $\OO_X^\perp$ and categorical Torelli statements formulated in terms of $\mathcal{K}u(X)$.

The first direction concerns smooth Fano varieties for which the anticanonical model is quadratically defined. The project \cite{LZ26a} studies whether multiplication in $\HS(\OO_X^\perp)$ recovers the quadratic equations of the anticanonical model and hence reconstructs $X$ in arbitrary dimension. This extends the present problem to a higher-dimensional class for which the equations are expected to be recovered directly from the orthogonal category.

A second direction develops Hochschild--Serre reconstruction for conic bundles and quadric fibrations through their Clifford components \cite{LZ26b}. Here the reconstruction problem must retain the linear structure over the base whenever it is not intrinsic to the abstract category. The aim is to recover the discriminant and the relevant Clifford data, and then to use these data for the corresponding fibrations and complete intersections of quadrics.

Finally, the same program is being developed for complete intersections of hypersurfaces. The $(2,3)$-complete-intersection case in $\PP^5$, together with further examples, is considered in \cite{LZ26c}. A complementary construction based on categorical Griffiths residues and the Cayley trick is pursued in \cite{LZ26d}. The intended comparison is between the Hochschild--Serre algebra of the residual category and the Jacobian-type algebra attached to the complete intersection; establishing multiplicativity and invariance under enhanced equivalences is part of these works in preparation.

\subsection{Organization}
Section 2 recalls the categorical and geometric preliminaries, including Hochschild--Serre algebras, matrix factorizations, Clifford components, abstract spinors, root stacks, and canonical algebras. Section 3 proves Theorem~\ref{thm:main} by Hochschild--Serre reconstruction. Section 4 gives the equivariant proof in degrees one and two. Section 5 constructs the spinor partners and identifies their canonical adjoint images. Section 6 compares spinor modification with hyperbolic equivalence and describes the five pairs of quartic conic-bundle classes. Section 7 translates the construction into weighted projective lines and canonical algebras.
\subsection{Acknowledgments}
Marco Rampazzo is supported by the Tsinghua-YMSC--Imperial College joint postdoctoral program in algebraic geometry. Shizhuo Zhang is supported by the Sun Yat-sen University Starting Grant 34000-12256019.

\subsection{AI disclosure}
The categorical Torelli theorems in degrees one, two, and three were proved entirely by the authors. Artificial intelligence (AI) was subsequently used to review the arguments, assist with verification, and identify and correct typographical errors.

In degree four, AI drew our attention to the equivalence between $\OO_S^\perp$ and the derived category of a weighted projective line. The authors then computed the diagonal subalgebra of the Hochschild--Serre algebra and proved the reconstruction theorem by this method.

AI also suggested studying del Pezzo surfaces of degree four as conic bundles. The authors constructed the spinor modifications and proved the identification of the abstract spinor bundle, with the normalization specified in the text, with the gluing object obtained from the exceptional complement by the right adjoint functor. This identification was inspired by earlier work on categorical Torelli for Gushel--Mukai threefolds \cite{JLLZ24}.

AI also assisted in organizing the representation-theoretic comparison in Section 7 and checking its compatibility with the Clifford-order description. The authors formulated the resulting statement and take responsibility for the final text and the correctness of its mathematical arguments.

\section{Preliminaries on derived categories, Clifford components, and stacky lines}
\subsection{Enhancements, orthogonals, and mutations}
Throughout the main text the base field is $\C$, except in Theorem~\ref{thm:hyperbolic}. All triangulated categories are homotopy categories of pretriangulated, idempotent-complete dg categories, and all equivalences are linear enhanced Morita equivalences over the stated base field. Functors between enhanced categories are understood as derived quasi-functors, or equivalently perfect bimodules in the smooth proper case. Their composition is derived tensor product. These conventions specify the functor categories in \eqref{eq:HSintro} and the equivariant constructions.

For an exceptional object $E\in\mathcal T$, we use the right-orthogonal convention
\[
E^\perp=\{F\in\mathcal T\mid\RHom(E,F)=0\}.
\]
The left mutation through $E$ is defined by the triangle
\[
\RHom(E,F)\otimes E\longrightarrow F\longrightarrow\LL_E(F),
\]
and the right mutation by
\[
\RR_E(F)\longrightarrow F\longrightarrow\RHom(F,E)^\vee\otimes E.
\]
If $\mathcal T=\langle\AA,E\rangle$, the inclusion $j:\AA\hookrightarrow\mathcal T$ has both adjoints. In particular $j^!E$ records the functor $G\mapsto\RHom(G,E)$ on $\AA$; this adjoint image will later become a spinor bundle.

If $\mathcal T$ is smooth and proper, its Serre functor is characterized by
\[
\Hom(F,G)^\vee\cong\Hom(G,S_{\mathcal T}F).
\]
An enhanced equivalence intertwines Serre functors. A chosen root of a Serre functor, a base-linear action, and a t-structure require a separate compatibility statement whenever they are used.

\subsection{The Hochschild--Serre algebra}
\begin{definition}\label{def:HS}
For a smooth proper dg category $\AA$, define
\[
\HS^{p,q}(\AA)=\Hom_{\Fun(\AA,\AA)}(\id_{\AA},S_{\AA}^q[p]),
\qquad \HS(\AA)=\bigoplus_{p,q\in\Z}\HS^{p,q}(\AA).
\]
For $a:\id\to S^q[p]$ and $b:\id\to S^{q'}[p']$, set
\[
b\star a:=S^q(b)[p]\circ a:\id\longrightarrow S^{q+q'}[p+p'].
\]
\end{definition}
The formula fixes the order of multiplication. We use the coherent powers of the Serre bimodule, including its inverse; the shifts have the usual Koszul convention. Conjugation by an invertible bimodule, with the induced coherent identifications of powers, gives an isomorphism of bigraded algebras \cite[Theorem 3.4]{LZ24}. No choice of a geometric polarization is included in this invariant.

For consistency checks we also use Hochschild homology. A complex del Pezzo surface of degree $d$ has Picard rank $10-d$. The Hochschild--Kostant--Rosenberg isomorphism gives
\[
\dim\HH_0(S)=12-d,\qquad \HH_i(S)=0\quad(i\neq0).
\]
The decomposition $\Db(S)=\langle\AA_S,\OO_S\rangle$ therefore yields
\begin{equation}\label{eq:HHdegree}
\dim\HH_0(\AA_S)=11-d.
\end{equation}

\subsection{Anticanonical models and matrix factorizations}
For $d=1,2,3$ the anticanonical model is a quasi-smooth weighted hypersurface
\[
S_d=\{W_D=0\}\subset\PP(a_0,a_1,a_2,a_3)
\]
with the following weights and degrees:
\begin{center}
\begin{tabular}{cccc}
\toprule
$d$&$(a_0,a_1,a_2,a_3)$&$D$&$\sum_i a_i$\\\midrule
1&$(1,1,2,3)$&6&7\\
2&$(1,1,1,2)$&4&5\\
3&$(1,1,1,1)$&3&4\\\bottomrule
\end{tabular}
\end{center}
The weighted ambient spaces are regarded as smooth quotient stacks when applying derived-category results. The hypersurfaces avoid their nontrivial stabilizer loci, so their derived categories are those of the smooth del Pezzo schemes. Let $P=\C[x_0,x_1,x_2,x_3]$, with $\deg x_i=a_i$. The potential $W_D\in P_D$ has an isolated critical point at the origin and $\sum_i a_i=D+1$. The graded Gorenstein ring $P/(W_D)$ has parameter one. By \cite[Theorem 2.5, Corollary 2.18, and Theorem 3.10]{Orl09}, and mutating the exceptional structure sheaf to the right, we obtain
\begin{equation}\label{eq:Orlov}
\Db(S_d)=\langle\AA_S,\OO_{S_d}\rangle,\qquad
\AA_S\simeq\operatorname{MF}^{\mathrm{gr}}(W_D).
\end{equation}
Let $\tau=(1)$ be internal grade shift on the enhanced graded matrix-factorization category. Then
\begin{equation}\label{eq:Serre}
\tau^D\cong[2],\qquad S_{\AA_S}\cong\tau^{-1}[2],\qquad
S_{\AA_S}^D\cong[2D-2].
\end{equation}
In particular $\tau\cong S_{\AA_S}^{-1}[2]$ is intrinsic here. Geometrically it corresponds to tensoring with the anticanonical polarization and left mutating through $\OO_S$.

We recall the additive sector formula. For $g\in\mu_D$, put
\[
I_g=\{i\mid g^{a_i}\neq1\},\qquad r_g=|I_g|,\qquad
k_g=-\sum_{i\in I_g}a_i,
\]
and let $W_g$ be the restriction of $W_D$ to the fixed locus of $g$. In the conventions of \cite[Theorem 5.39]{BFK14}, as unpacked in \cite[Theorem 3.5 and Remark 3.6]{LRZ24},
\begin{equation}\label{eq:BFK}
\Hom(\Delta,\Delta(t)[m])\cong
\bigoplus_{\substack{g\in\mu_D\\r_g\equiv m\ (2)}}
\Jac(W_g)_{t-k_g+D(m-r_g)/2}.
\end{equation}
Here $\Delta$ is the diagonal bimodule. For $\HS^{p,q}$, equation \eqref{eq:Serre} gives
\[t=-q,\qquad m=2q+p.\]
The summand $g=1$ is the identity sector. We will use multiplication only on the diagonal $m=0$, or $(p,q)=(2e,-e)$. On this diagonal, identity-sector classes are multiplication maps by homogeneous polynomials; their products are the ordinary Jacobian products \cite[Remark 3.7]{LRZ24}. This specifies the multiplicative comparison actually needed below. We do not identify the full, periodic Hochschild--Serre algebra with a single copy of the Jacobian ring.

\subsection{Conic bundles and Clifford components}\label{subsec:clifford}
Let $f:S\to B$ be a flat conic bundle over a smooth curve, presented by a line-bundle-valued quadratic form $q:L\to\Sym^2E^\vee$, with $\rk E=3$. Write $\mathcal B_0=\Cl_0(q)$ for its even Clifford algebra. By \cite[Theorem 4.2]{Kuz08}, there is a Fourier--Mukai embedding and a semiorthogonal decomposition
\begin{equation}\label{eq:quadric}
\Db(S)=\langle\Db(B,\mathcal B_0),f^*\Db(B)\rangle.
\end{equation}
We choose the embedding whose image is $\Ker(Rf_*)=:\Kf$. For $B=\PP^1$ this gives
\begin{equation}\label{eq:conicsod}
\Db(S)=\langle\Kf,E_f,\OO_S\rangle,\qquad
\AA_S=\langle\Kf,E_f\rangle,\qquad E_f=f^*\OO_B(-1).
\end{equation}
Indeed $R\Gamma(B,\OO_B(-1))=0$, and for $G\in\Kf$,
\[\RHom_S(\OO_S,G)=R\Gamma(B,Rf_*G)=0.\]
The standard t-structure of $\Db(S)$ restricts to $\Kf$ by \cite[Lemma 2.9]{Kuz25}. This is the t-structure used for spinor equivalences.
\begin{definition}\label{def:spinor}
An abstract spinor bundle on $f:S\to B$ is a rank-two vector bundle $\mathcal F$ such that
\[
Rf_*\mathcal F=0,\qquad c_1(\mathcal F)=K_{S/B}\text{ in }\Pic(S)/f^*\Pic(B).
\]
\end{definition}
Its relative endomorphism algebra $f_*\cEnd(\mathcal F)$ is pointwise Clifford. The construction of \cite[Theorem 1.2]{Kuz25} associates a conic bundle $S_{\mathcal F}\to B$, called the spinor modification.
\begin{theorem}\label{thm:kuzspinor}
By \cite[Theorem 1.2]{Kuz25}, an abstract spinor bundle $\mathcal F$ determines a flat conic bundle $S_{\mathcal F}/B$ with
\[\Cl_0(S_{\mathcal F}/B)\cong f_*\cEnd(\mathcal F).\]
Its even Clifford algebra is $B$-linearly Morita equivalent to that of $S/B$, and there is a $B$-linear t-exact Fourier--Mukai equivalence
\[\Ker(S_{\mathcal F}/B)\xrightarrow{\sim}\Ker(S/B).\]
It sends the canonical spinor of the reconstructed quadratic presentation to $\mathcal F$. Conversely, a $B$-linear Morita equivalence of even Clifford algebras, or a $B$-linear t-exact Fourier--Mukai equivalence of kernel categories, realizes the other conic bundle as a spinor modification.
\end{theorem}
We fix the following normalization. A conic bundle is presented by its relative anticanonical embedding, so that $\OO_{S/B}(1)=\omega_{S/B}^{-1}$ and $\det E\otimes L\cong\OO_B$. The normalized canonical spinor $\mathcal F^0_{S/B}$ is the unique acyclic extension
\begin{equation}\label{eq:canonicalintro}
0\longrightarrow\omega_{S/B}\longrightarrow\mathcal F^0_{S/B}
\longrightarrow\OO_S\longrightarrow0,
\end{equation}
whose connecting morphism $\OO_B\to R^1f_*\omega_{S/B}\cong\OO_B$ is the identity. For an arbitrary presentation, \cite[Lemma 2.16]{Kuz25} distinguishes this extension from the canonical Clifford-module spinor by the base twist $f^*(\det E\otimes L)$. They agree in the normalization just fixed. All uses of the canonical spinor below refer to this normalization.

\subsection{Root stacks, weighted projective lines, and canonical algebras}
Let $D=t_1+\cdots+t_r\subset\PP^1$ be reduced. The root stack
\[\mathcal X_D:=\rootstack{\PP^1}{D}\]
is the weighted projective line with weight two at every $t_i$. Locally at a marked point it has the presentation
\[[\Spec k[[u]]/\mu_2],\qquad t=u^2.\]
A split hereditary order of ramification index two is locally Morita equivalent to
\begin{equation}\label{eq:Iwahori}
\begin{pmatrix}\C[[t]]&\C[[t]]\\t\C[[t]]&\C[[t]]\end{pmatrix}.
\end{equation}
The order--stack correspondence identifies its coherent module category with $\coh\mathcal X_D$ \cite{CI04}. In our situation this correspondence is untwisted: the generic conic over $\C(t)$ is split, and the Brauer group of the root stack vanishes. Indeed the coarse curve has trivial Brauer group and the possible residue groups are $H^1(\Spec\C,\Z/2)=0$. Thus for a conic bundle with simple discriminant $D$,
\begin{equation}\label{eq:rootcliff}
\Kf\simeq\Db(\PP^1,\mathcal B_0)\simeq\Db(\coh\mathcal X_D).
\end{equation}
Over a nonclosed field the stack model may carry an Azumaya twist; such descent data are outside the complex reconstruction statements.

For weights $\mathbf p=(p_1,\ldots,p_r)$, the Geigle--Lenzing lattice and dualizing element are
\[
\mathbb L(\mathbf p)=\langle\vec c,\vec x_1,\ldots,\vec x_r\mid p_i\vec x_i=\vec c\rangle,
\qquad \vec\omega=(r-2)\vec c-\sum_i\vec x_i.
\]
The canonical tilting bundle is
\begin{equation}\label{eq:tilting}
T_{\mathrm{can}}=\OO\oplus\bigoplus_{i=1}^r\bigoplus_{j=1}^{p_i-1}\OO(j\vec x_i)\oplus\OO(\vec c).
\end{equation}
We use right modules and put $\Lambda(\mathbf p;D)=\End(T_{\mathrm{can}})$, with composition as multiplication. Then $\RHom(T_{\mathrm{can}},-)$ takes values in right $\Lambda$-modules, and \cite{GL87} gives
\begin{equation}\label{eq:GL}
\Db(\coh\mathcal X(\mathbf p;D))\simeq\Db(\modu\text{-}\Lambda(\mathbf p;D)).
\end{equation}
Using the opposite algebra instead would require left modules. An honest conic bundle produces only weights two: its even Clifford algebra is a quaternion order, so every simple ramification index is two.

\subsection{Equivariant categories and root-stack quotients}
Let a finite group $G$ act coherently on an enhanced category $\AA$. We write $\AA^G$ for the equivariant category. An equivalence supplied with a coherent intertwiner induces an equivalence of equivariant categories; see \cite[Theorem 6.10]{Ela14}. An isomorphism between the underlying autoequivalences alone does not specify that coherence. It will be checked in Section 4.

For a double cover $X\to Y$ branched along a smooth Cartier divisor $C$, with deck group $\mu_2$, one has
\begin{equation}\label{eq:quotient}
[X/\mu_2]\cong\rootstack{Y}{C}.
\end{equation}
This also holds for a smooth Deligne--Mumford base stack $Y$ and a representable double cover. It follows locally from $w^2=t$ and the root-stack definition \cite[Section 2]{Cad07}. By \cite[Theorem 6.11 and Corollary 6.12]{BS20}, the derived category has a decomposition with components $\Db(C)$ and $\Db(Y)$:
\begin{equation}\label{eq:rootSOD}
\Db([X/\mu_2])=\langle\Db(C),\Db(Y)\rangle,
\end{equation}
with the appropriate character twists in the embeddings. Equivariantizing a preserved decomposition $\Db(X)=\langle\AA,\OO_X\rangle$ instead gives
\begin{equation}\label{eq:eqSOD}
\Db([X/G])=\langle\AA^G,\langle\OO_X\rangle^G\rangle,
\qquad \langle\OO_X\rangle^{\mu_2}\simeq\Db(\Rep\mu_2),
\end{equation}
where the last category has two orthogonal exceptional generators \cite[Theorem 6.3]{Ela12}.

\section{Categorical Torelli theorem via Hochschild--Serre algebra}
We now prove Theorem~\ref{thm:main}.
\subsection{Proof of categorical Torelli for del Pezzo surfaces of degree at most four}
\begin{lemma}\label{lem:product}
Under \eqref{eq:BFK}, the identity sectors in
\[
\bigoplus_{e\geq0}\Hom(\Delta,\Delta(e))
=\bigoplus_{e\geq0}\HS^{2e,-e}(\AA_S)
\]
form a graded subalgebra isomorphic to $\Jac(W_D)$.
\end{lemma}
\begin{proof}
Multiplication by a homogeneous polynomial $a\in P_e$ gives a morphism from the diagonal factorization to its twist by $e$. In the Koszul model of the diagonal, multiplication by each $\partial_iW_D$ is null-homotopic. Thus these maps factor through $\Jac(W_D)_e$. The additive kernel calculation identifies their classes with the identity sector in \eqref{eq:BFK}. Composition of multiplication maps by $a$ and $b$ is multiplication by $ab$. The twists do not change these polynomial representatives, and the cohomological kernel degree is zero, so no additional sign enters. This is the multiplicative comparison stated in \cite[Remark 3.7]{LRZ24}. Since the identity-sector map is injective in every degree by \eqref{eq:BFK}, the conclusion follows.
\end{proof}
\begin{lemma}\label{lem:generate}
Suppose $e_1,\ldots,e_s>0$ satisfy
\[
\HS^{2e_i,-e_i}(\AA_S)=\Jac(W_D)_{e_i}
\]
with no nonidentity-sector summand. If these graded pieces generate $\Jac(W_D)$, then
\[
\langle\HS^{2e_1,-e_1}(\AA_S),\ldots,\HS^{2e_s,-e_s}(\AA_S)\rangle\alg
\cong\Jac(W_D)
\]
as graded algebras, where the bidegree $(2e,-e)$ has algebra degree $e$.
\end{lemma}
\begin{proof}
By Lemma~\ref{lem:product}, every product in the indicated spaces is its ordinary Jacobian product. Their image is the full Jacobian algebra by the generation hypothesis; injectivity holds in each degree by the additive sector formula. Although an arbitrary equivalence need not preserve the entire sector decomposition, it preserves these generating bidegrees and their products. The generated subalgebra is therefore intrinsic.
\end{proof}

\subsection{Degree three}
Since a cubic surface has index one, its Kuznetsov component is precisely $\AA_S=\OO_S^\perp$. The categorical Torelli theorem in this case was established in \cite[Corollary 1.6]{LZ24}. We give the comparison adapted to our conventions.

Let
\[
S=\{f_3(x_0,x_1,x_2,x_3)=0\}\subset\PP^3,\qquad
S'=\{f'_3(x'_0,x'_1,x'_2,x'_3)=0\}\subset\PP^3
\]
be smooth cubic surfaces, and let $\Phi:\AA_S\xrightarrow{\sim}\AA_{S'}$ be an enhanced equivalence. Set
\[
R_3=\frac{\C[x_0,x_1,x_2,x_3]}{(\partial_{x_0}f_3,\ldots,\partial_{x_3}f_3)},\qquad
R'_3=\frac{\C[x'_0,x'_1,x'_2,x'_3]}{(\partial_{x'_0}f'_3,\ldots,\partial_{x'_3}f'_3)}.
\]
For a finite-dimensional graded algebra $T=\bigoplus_{m\geq0}T_m$, write
\[H_T(z)=\sum_{m\geq0}(\dim_\C T_m)z^m.\]
The four partial derivatives of either cubic form a regular sequence of quadrics. Consequently
\[
H_{R_3}(z)=H_{R'_3}(z)=\left(\frac{1-z^2}{1-z}\right)^4=(1+z)^4,
\]
and the dimensions in degrees zero through four are $(1,4,6,4,1)$.

For either nontrivial $g\in\mu_3$, all four coordinates move, so $r_g=4$, $k_g=-4$, and $W_g=0$. If $s=(3p+4q)/2$, formula \eqref{eq:BFK} gives
\begin{equation}\label{eq:cubicHS}
\HS^{p,q}(\AA_S)\cong
\begin{cases}
(R_3)_s\oplus\C^2,&p\text{ even and }s=2,\\
(R_3)_s,&p\text{ even and }s\neq2,\\
0,&p\text{ odd},
\end{cases}
\end{equation}
where $(R_3)_s=0$ outside $\{0,1,2,3,4\}$. The identical formula holds for $S'$.
\begin{proposition}\label{prop:cubic}
The equivalence $\Phi$ induces an isomorphism $R_3\cong R'_3$ of graded algebras and consequently an isomorphism $S\cong S'$.
\end{proposition}
\begin{proof}
Conjugation by $\Phi$ gives an isomorphism of bigraded algebras
\[\Phi_{\HS}:\HS(\AA_S)\xrightarrow{\sim}\HS(\AA_{S'}).\]
The identity-sector degree in \eqref{eq:cubicHS} is $-q+\tfrac32(2q+p)=s$. A nontrivial sector has degree $s-2$ in $\Jac(0)=\C$, so it occurs exactly when $s=2$. At $(p,q)=(2,-1)$ we have $s=1$, and hence
\[
V:=\HS^{2,-1}(\AA_S)=(R_3)_1,\qquad
V':=\HS^{2,-1}(\AA_{S'})=(R'_3)_1.
\]
The restriction $\Phi_{2,-1}:V\to V'$ is an isomorphism. For the multiplication maps
\[
\mu_2:\Sym^2V\longrightarrow\HS^{4,-2}(\AA_S),\qquad
\mu'_2:\Sym^2V'\longrightarrow\HS^{4,-2}(\AA_{S'}),
\]
multiplicativity gives
\[
\Phi_{4,-2}\circ\mu_2=\mu'_2\circ\Sym^2(\Phi_{2,-1}),\qquad
\Sym^2(\Phi_{2,-1})(\Ker\mu_2)=\Ker\mu'_2.
\]
By Lemma~\ref{lem:product}, the products lie in the identity sector, even though the target also has two twisted-sector classes. Therefore
\[
\Ker\mu_2=\langle\partial_{x_0}f_3,\ldots,\partial_{x_3}f_3\rangle,\qquad
\Ker\mu'_2=\langle\partial_{x'_0}f'_3,\ldots,\partial_{x'_3}f'_3\rangle
\]
inside the respective quadratic polynomial spaces. Applying the same comparison in every degree and using generation in degree one gives
\[
R_3\cong\langle\HS^{2,-1}(\AA_S)\rangle\alg
\xrightarrow[\Phi_{\HS}]{\sim}
\langle\HS^{2,-1}(\AA_{S'})\rangle\alg\cong R'_3.
\]
By \cite[Theorem 3.2 and Corollary 3.3]{Ahm10}, the graded Milnor algebra identifies homogeneous equations of the same degree up to a linear change of coordinates. Allowing a nonzero scalar does not change the projective zero locus. Thus $S\cong S'$.
\end{proof}

\subsection{Degree two}
Completing the square gives
\[
\begin{gathered}
S_2=\{W=y^2-f_4(x_0,x_1,x_2)=0\}\subset\PP(1,1,1,2),\\
R_2=\Jac(W)\cong\frac{\C[x_0,x_1,x_2]}{(\partial f_4)}.
\end{gathered}
\]
Its Hilbert series is $(1+z+z^2)^3$, with Hilbert function $(1,3,6,7,6,3,1)$. The sector data are
\begin{center}
\begin{tabular}{ccccc}
\toprule
$g$&$I_g$&$r_g$&$k_g$&$W_g$\\\midrule
1&$\varnothing$&0&0&$W$\\
$-1$&$\{x_0,x_1,x_2\}$&3&$-3$&$y^2$\\
$i,-i$&$\{x_0,x_1,x_2,y\}$&4&$-5$&0\\\bottomrule
\end{tabular}
\end{center}
\begin{proposition}\label{prop:degree2}
For $s=2p+3q$,
\[
\HS^{p,q}(\AA_S)\cong
\begin{cases}
(R_2)_s\oplus\C^2,&p\text{ even and }s=3,\\
(R_2)_s,&p\text{ even and }s\neq3,\\
\C,&p\text{ odd and }s=3,\\
0,&p\text{ odd and }s\neq3.
\end{cases}
\]
Consequently $\langle\HS^{2,-1}(\AA_S)\rangle\alg\cong R_2$.
\end{proposition}
\begin{proof}
The identity degree is $-q+2(2q+p)=s$. The $(-1)$-sector has odd parity and degree $s-3$ in $\Jac(y^2)=\C$. The primitive fourth-root sectors have even parity and the same shifted degree in $\Jac(0)=\C$. At $(p,q)=(2,-1)$ the degree is one, so no twisted sector contributes and $\HS^{2,-1}=(R_2)_1$. The ring is generated by this piece, and Lemma~\ref{lem:generate} applies.
\end{proof}
By \cite[Theorem 3.2 and Corollary 3.3]{Ahm10}, the graded Milnor algebra reconstructs the smooth plane quartic $\{f_4=0\}$. The del Pezzo surface is the double plane branched over that quartic, using the unique square root $\OO_{\PP^2}(2)$ of its divisor line bundle. Thus it is reconstructed as well.

\subsection{Degree one and the missing weight-two generator}
Let $x_0,x_1,y,z$ have weights $1,1,2,3$. A smooth degree-one del Pezzo surface has the normal form
\begin{equation}\label{eq:sextic}
S_1=\{W=z^2+y^3+yf_4(x_0,x_1)+f_6(x_0,x_1)=0\}
\subset\PP(1,1,2,3).
\end{equation}
Put $G_6=y^3+yf_4+f_6$. Then
\[
R_1=\Jac(W)\cong
\frac{\C[x_0,x_1,y]}{(\partial_{x_0}G_6,\partial_{x_1}G_6,\partial_yG_6)}.
\]
The partial derivatives form a regular sequence of weights $5,5,4$. Hence
\[
H_{R_1}(z)=\frac{(1-z^5)^2(1-z^4)}{(1-z)^2(1-z^2)},
\]
with Hilbert function $(1,2,4,6,8,8,8,6,4,2,1)$. For a primitive sixth root $\zeta$, the nontrivial sectors are
\begin{center}
\begin{tabular}{ccccc}
\toprule
$g$&fixed coordinates&$r_g$&$k_g$&$W_g$\\\midrule
$-1$&$y$&3&$-5$&$y^3$\\
$\zeta^2,\zeta^4$&$z$&3&$-4$&$z^2$\\
$\zeta,\zeta^5$&origin&4&$-7$&0\\\bottomrule
\end{tabular}
\end{center}
Here $\Jac(y^3)=\C[y]/(y^2)$ with $\deg y=2$.
\begin{proposition}\label{prop:degree1}
Put $s=3p+5q$. Then
\[
\HS^{p,q}(\AA_S)\cong
\begin{cases}
(R_1)_s\oplus\C^2,&p\text{ even and }s=5,\\
(R_1)_s,&p\text{ even and }s\neq5,\\
\C,&p\text{ odd and }s\in\{4,6\},\\
\C^2,&p\text{ odd and }s=5,\\
0,&\text{otherwise}.
\end{cases}
\]
Moreover,
\begin{equation}\label{eq:genweight2}
\langle\HS^{2,-1}(\AA_S),\HS^{4,-2}(\AA_S)\rangle\alg\cong R_1.
\end{equation}
\end{proposition}
\begin{proof}
The identity degree is $-q+3(2q+p)=s$. In the three rows of the table, the shift of this degree is respectively $-4,-5,-5$. The parity conditions and the supports of $\Jac(y^3)$, $\Jac(z^2)$, and $\Jac(0)$ give the formula. At the two relevant bidegrees we obtain
\[
\HS^{2,-1}=(R_1)_1,\qquad \HS^{4,-2}=(R_1)_2,
\]
with no twisted contribution. In coordinates, $(R_1)_1=\langle x_0,x_1\rangle$ and $(R_1)_2=\Sym^2(R_1)_1\oplus\C y$. Although this splitting depends on coordinates, the one-dimensional quotient
\[
\frac{\HS^{4,-2}}{\Image(\Sym^2\HS^{2,-1}\longrightarrow\HS^{4,-2})}
\]
is intrinsic. It detects the missing weight-two generator. The two pieces generate $R_1$, so Lemma~\ref{lem:generate} proves \eqref{eq:genweight2}.
\end{proof}
The normalized weighted Euler identity is
\[
6G_6=x_0\partial_{x_0}G_6+x_1\partial_{x_1}G_6+2y\partial_yG_6.
\]
Thus the Milnor and Tjurina algebras agree. More precisely, the graded isomorphism obtained above is between Milnor algebras of sextics with the fixed weights $(1,1,2)$. By \cite[Theorem 2 and Corollary 3]{Ahm12}, it gives a weight-preserving polynomial change of coordinates identifying the two equations. In this case an isomorphism on the minimal generators has the form
\[
(x_0,x_1)\longmapsto A(x_0,x_1),\qquad y\longmapsto c y+q_2(x_0,x_1),
\]
with $A\in\operatorname{GL}_2(\C)$ and $c\neq0$, so it lifts to a graded polynomial automorphism. Adjoining the variable $z$ and rescaling it if necessary identifies the double covers in \eqref{eq:sextic}. The second bidegree in \eqref{eq:genweight2} is necessary: products of the degree-one piece do not generate the class of $y$.

\subsection{Degree four and the orbifold canonical subalgebra}
Let $S_4=Q_0\cap Q_1\subset\PP^4$. Its pencil has five distinct singular members, defining a reduced divisor $D=\{\lambda_1,\ldots,\lambda_5\}\subset\PP^1$. The quadric-fibration and order--stack descriptions give
\begin{equation}\label{eq:quarticcat}
\AA_S\simeq\Db(\coh\mathcal X_D),\qquad
\mathcal X_D=\rootstack{\PP^1}{D};
\end{equation}
see \cite{Kuz08,CI04} and \cite[Section 2.4]{Ela26}. Write $\mathcal X=\mathcal X_D$, $B=\PP^1$, and let $\pi:\mathcal X\to B$ be the coarse map. Denote by $D_i$ the reduced stacky divisor over $\lambda_i$, so that $2D_i=\pi^*(\lambda_i)$. We first justify the Serre functor and the multiplicative comparison needed for reconstruction.

\begin{lemma}\label{lem:quarticSerre}
By \cite[Section 1.2 and Serre duality 2.2]{GL87}, the canonical bundle and the Serre functor of $\Db(\coh\mathcal X)$ are
\[
\omega_{\mathcal X}\cong\pi^*\omega_B\otimes
\OO_{\mathcal X}\Bigl(\sum_{i=1}^5D_i\Bigr),\qquad
S_{\mathcal X}(E)\cong E\otimes\omega_{\mathcal X}[1].
\]
For every line bundle $L$ on $\mathcal X$, multiplication by a section induces an isomorphism
\[
H^0(\mathcal X,L)\xrightarrow{\sim}
\Hom_{\Fun(\Perf(\mathcal X),\Perf(\mathcal X))}
(\id,-\otimes L).
\]
These isomorphisms are compatible with tensor products of line bundles and composition of natural transformations. Consequently, \eqref{eq:quarticcat} induces an isomorphism of graded algebras
\begin{equation}\label{eq:canonicalring}
R_4:=\bigoplus_{p\geq0}\HS^{-p,p}(\AA_S)
\cong\bigoplus_{p\geq0}H^0(\mathcal X_D,\omega_{\mathcal X_D}^{p}).
\end{equation}
Here an element of $H^0(\mathcal X,\omega_{\mathcal X}^p)$ has degree $p$ and corresponds to the Hochschild--Serre bidegree $(-p,p)$.
\end{lemma}
\begin{proof}
We prove the assertions about natural transformations and their products. Let $\Delta:\mathcal X\to\mathcal X\times\mathcal X$ be the diagonal. The standard quotient presentation of a weighted projective line makes $\mathcal X$ a perfect stack; see \cite[Section 1]{GL87} and \cite[Corollary 3.22]{BFN10}. The integral-kernel comparison, restricted to the functors preserving perfect complexes, identifies
\[
\Hom_{\Fun(\Perf(\mathcal X),\Perf(\mathcal X))}(\id,-\otimes L)
\cong
\Hom_{\Db(\coh(\mathcal X\times\mathcal X))}
(\Delta_*\OO_{\mathcal X},\Delta_*L);
\]
see \cite[Corollary 4.10 and Remark 4.11]{BFN10}. The two kernels on the right are coherent sheaves in degree zero. We may therefore compute their degree-zero derived Hom by ordinary sheaf morphisms.

It remains to check this sheaf Hom at the stacky points. Use the completed chart $[\Spec A/\mu_2]$, where $A=\C[[z]]$, the coarse coordinate is $t=z^2$, and the stabilizer acts by $z\mapsto-z$. On the product of the local covers the pullback of the diagonal kernel is the direct sum of the structure sheaves of the two graphs
\[
\Gamma_+=\{z_2=z_1\},\qquad \Gamma_-=\{z_2=-z_1\}.
\]
The target kernel has the corresponding line-bundle twists on these graphs. A morphism from the $\Gamma_+$ summand to the $\Gamma_-$ summand is, after trivializing that line bundle, determined by an element $a\in\C[[z_1]]$ satisfying
\[
(z_2-z_1)a\big|_{\Gamma_-}=-2z_1a=0.
\]
Hence $a=0$, and the same argument applies with the two graphs interchanged. The remaining morphisms are multiplication by sections on the individual graphs. The $\mu_2\times\mu_2$-equivariance permutes the graphs and identifies these sections; the stabilizer of the identity graph is the diagonal $\mu_2$, so descent leaves precisely the invariant sections of $L$. At an ordinary point the usual diagonal calculation gives the same conclusion. Thus, globally,
\[
\Hom_{\Db(\coh(\mathcal X\times\mathcal X))}
(\Delta_*\OO_{\mathcal X},\Delta_*L)
\cong H^0(\mathcal X,L).
\]
The resulting natural transformation associated with $s\in H^0(\mathcal X,L)$ is exactly $\eta_s(E)=\id_E\otimes s$. The local vanishing just proved uses the effective stabilizer action on the coordinate; without it an arbitrary stack may have additional degree-zero transformations.

Finally, choose the coherent identifications
\[
S_{\mathcal X}^p[-p]\cong-\otimes\omega_{\mathcal X}^p.
\]
For $s\in H^0(\mathcal X,\omega_{\mathcal X}^p)$ and $t\in H^0(\mathcal X,\omega_{\mathcal X}^q)$, the product of Definition~\ref{def:HS} becomes
\[
\eta_t\star\eta_s
=S_{\mathcal X}^p(\eta_t)[-p]\circ\eta_s
=\eta_{st}.
\]
After cancellation of the shifts these are degree-zero multiplication maps, so this is ordinary section multiplication. This proves the algebra isomorphism \eqref{eq:canonicalring}, rather than only an identification of its graded pieces. Transport through \eqref{eq:quarticcat} is compatible with multiplication by the Morita invariance of the Hochschild--Serre algebra.
\end{proof}

We now compute this section ring. The same local chart gives
\[
\pi_*\OO_{\mathcal X}\Bigl(p\sum_iD_i\Bigr)
\cong\OO_B\bigl(\lfloor p/2\rfloor D\bigr).
\]
Indeed, the local module is $A z^{-p}$; its invariant generator is $t^{-m}$ for both $p=2m$ and $p=2m+1$. The canonical-bundle formula and the projection formula therefore give
\begin{equation}\label{eq:pushomega}
\pi_*\omega_{\mathcal X}^p
\cong\omega_B^p\otimes\OO_B\bigl(\lfloor p/2\rfloor D\bigr).
\end{equation}
Put $A_B=\omega_B^2\otimes\OO_B(D)$. Then $\deg A_B=1$, $\omega_{\mathcal X}^2\cong\pi^*A_B$, and
\[
\pi_*\omega_{\mathcal X}^{2m}\cong A_B^m,\qquad
\pi_*\omega_{\mathcal X}^{2m+1}\cong\omega_B\otimes A_B^m.
\]
In particular,
\[
\dim(R_4)_{2m}=m+1,\qquad
\dim(R_4)_{2m+1}=\max\{m-1,0\}\qquad(m\geq0).
\]
The ring in \eqref{eq:canonicalring} is the orbifold canonical section ring of $\mathcal X_D$, graded by the tensor power of $\omega_{\mathcal X_D}$. We describe its weighted homogeneous coordinates as follows.

\begin{proposition}\label{prop:quarticring}
There are generators $u,v,w$ of degrees $2,2,5$, respectively, such that
\begin{equation}\label{eq:quarticring}
R_4\cong\frac{\C[u,v,w]}{(w^2-\ell_1(u,v)\cdots\ell_5(u,v))},
\end{equation}
where $\ell_i=0$ is the point $\lambda_i$.
\end{proposition}
\begin{proof}
Choose a basis $u,v$ of $H^0(B,A_B)$. Since $A_B\cong\OO_{\PP^1}(1)$, the even subalgebra is $\C[u,v]$, with $\deg u=\deg v=2$. Moreover, $M:=\omega_B\otimes A_B^2\cong\OO_B$. A nonzero section of $M$ defines $w\in(R_4)_5$, and \eqref{eq:pushomega} gives
\[
(R_4)_{5+2k}=wH^0(B,A_B^k)\quad(k\geq0),\qquad
(R_4)_1=(R_4)_3=0.
\]
Thus $R_4=\C[u,v]\oplus w\C[u,v]$ as a graded module over its even subalgebra.

To determine $w^2$, use the local chart $t=z^2$ at $\lambda_i$. The invariant generators of $\pi_*\omega_{\mathcal X}^5$ and $\pi_*\omega_{\mathcal X}^{10}$ are $z(dz)^5$ and $(dz)^{10}$, respectively. Since the chosen section of $M$ is nowhere vanishing, locally $w=h(t)z(dz)^5$ for a unit $h(t)$, and hence
\[
w^2=h(t)^2t(dz)^{10}.
\]
Therefore $w^2\in H^0(B,A_B^5)$ has a simple zero at each $\lambda_i$. As $\deg A_B^5=5$, its zero divisor is exactly $D$, so
\[
w^2=c\,\ell_1(u,v)\cdots\ell_5(u,v),\qquad c\in\C^\times.
\]
Rescale $w$ to absorb $c$. The resulting surjective homomorphism in \eqref{eq:quarticring} is an isomorphism: both sides are free $\C[u,v]$-modules with basis $1,w$, and the homomorphism identifies these bases.
\end{proof}
We make explicit why this presentation reconstructs the discriminant intrinsically. From the graded algebra alone one recovers the two-dimensional space $V=(R_4)_2$, the line $(R_4)_5$, and the multiplication isomorphism
\[
\Sym^5V\xrightarrow{\sim}(R_4)_{10}.
\]
The square of any nonzero $w\in(R_4)_5$ defines a nonzero binary quintic $f\in\Sym^5V$, determined up to a scalar. Its zero divisor in $\PP(V^\vee)$ is precisely $D$. A change of basis of $V$ changes the coordinate on this projective line, and a change of generator of $(R_4)_5$ rescales $f$. Hence the graded algebra determines the unordered divisor $D$ modulo $\PGL_2$. By Lemma~\ref{lem:quarticSerre}, this algebra is recovered from the enhanced category without choosing a coarse-space action or a line bundle as additional categorical data.

Finally, choose a nonsingular member $Q_0$ of the pencil. The endomorphism $Q_0^{-1}Q_1$ is self-adjoint for $Q_0$ and has distinct eigenvalues because the discriminant is reduced. Its eigenvectors are $Q_0$-orthogonal. None has zero $Q_0$-norm, since it would then be orthogonal to an eigenbasis and lie in the radical of $Q_0$. Rescaling the eigenvectors over $\C$ gives
\[
Q_0=\sum_{i=1}^5x_i^2,\qquad Q_1=\sum_{i=1}^5\lambda_i x_i^2.
\]
The eigenvalues represent the five discriminant points in an affine coordinate for which $Q_0$ is the nonsingular member at infinity. Permuting them permutes the coordinates $x_i$. Changing the basis of the pencil acts on them by a projective linear transformation; choosing the first new form nonsingular and rescaling the eigenvectors again gives the same normal form for the transformed eigenvalues. Thus the unordered projective equivalence class of $D$ determines the intersection of quadrics up to isomorphism. This gives a Hochschild--Serre algebra proof of Elagin's degree-four reconstruction theorem \cite[Theorem 3.1]{Ela26}.

\subsection{Proof of categorical Torelli}\label{subsec:proofTorelli}
\begin{proof}[Proof of Theorem~\ref{thm:main}]
An enhanced equivalence preserves Hochschild homology, so \eqref{eq:HHdegree} gives $d=d'$. It also preserves the bigraded Hochschild--Serre algebra and its multiplication. Propositions~\ref{prop:cubic}, \ref{prop:degree2}, and \ref{prop:degree1} reconstruct the appropriate graded Jacobian algebra in degrees three, two, and one. The graded homogeneous reconstruction results then identify the anticanonical equations. In degree four, Proposition~\ref{prop:quarticring} recovers the discriminant of the pencil and hence the intersection of quadrics. Thus $S\cong S'$ in every degree $1\leq d\leq4$.
\end{proof}

\section{Equivariant reconstruction in degrees one and two}
This section gives a second proof for $d=1,2$. The Serre functor recovers the deck involution. Equivariantization and topological K-theory recover the polarized integral first cohomology of the branch curve. Classical Torelli reconstructs that curve, and its canonical model reconstructs the double cover.

If $K_S^2=2$, the anticanonical system defines a finite double cover $\pi:S\to\PP^2$ branched over a smooth plane quartic $C$, with Geiser involution $\gamma$. If $K_S^2=1$, the bi-anticanonical morphism is a double cover of a quadric cone. On the smooth weighted-projective stack $Y=\PP(1,1,2)$, it is the representable double cover
\[
S=\{z^2=G_6(x_0,x_1,y)\}\subset\PP(1,1,2,3),\qquad
C=\{G_6=0\}\subset Y.
\]
Its deck transformation $z\mapsto-z$ is the Bertini involution $\beta$; see \cite[Section 2]{AB21}. The curve $C$ avoids the stacky point of $Y$, since the coefficient of $y^3$ is nonzero. Over that point the character of $\OO_Y(3)$ is nontrivial, so the stabilizer exchanges the two sheets of the relative spectrum. The total space is the smooth del Pezzo scheme. This distinguishes the stack double cover from its coarse morphism to the quadric cone, which has an isolated branch point at the vertex.

\begin{proposition}\label{prop:involution}
The deck involutions act on $\AA_S$ by
\[
S_{\AA_S}^3\cong\beta^*[5]\quad(d=1),\qquad
S_{\AA_S}^2\cong\gamma^*[3]\quad(d=2).
\]
Thus $\epsilon_1=S_{\AA_S}^3[-5]$ and $\epsilon_2=S_{\AA_S}^2[-3]$ are intrinsic involutions.
\end{proposition}
\begin{proof}
Let $b$ be half the weighted degree of the covering equation, so $b=3$ for $d=1$ and $b=2$ for $d=2$, and let $\delta$ be the corresponding deck involution. By \cite[Example 3.2, Proposition 3.4, and Corollary 3.18]{Kuz19}, the rotation functor satisfies $\tau^b\cong\delta^*[1]$; see also \cite[Theorem 7.7]{KP17}. The construction in \cite[Section 3.1]{Kuz19} explicitly allows smooth stacks. In the present case the pushforward of the double cover is spherical, its source twist is $\delta^*(-\otimes\OO_S(-b))[1]$, and its target twist is $(-\otimes\OO_Y(-b))[-1]$. The latter shifts the Lefschetz blocks by $-b$. The required collection has length four on $\PP(1,1,2)$ and length three on $\PP^2$, so the rotation identity applies in both cases. Formula \eqref{eq:Serre} therefore gives
\[
S_{\AA_S}^b\cong\tau^{-b}[2b]\cong\delta^*[2b-1].
\]
This gives exactly the two displayed identities. Their square relations are consistent with $S^6\cong[10]$ and $S^4\cong[6]$ from \eqref{eq:Serre}.
\end{proof}
Let $G=\mu_2$ and let $\Phi:\AA_S\simeq\AA_{S'}$ be an enhanced equivalence with $d=d'\in\{1,2\}$. It intertwines the Serre functors, hence the underlying involution functors $\epsilon_d$. We now check coherence with the geometric actions.

Choose an isomorphism $\eta:\Phi\epsilon_d\to\epsilon'_d\Phi$ and the geometric square isomorphisms $\alpha:\epsilon_d^2\to\id$ and $\alpha':(\epsilon'_d)^2\to\id$. The two maps from $\Phi\epsilon_d^2$ to $\Phi$ are
\[
\Phi(\alpha),\qquad
\alpha'\Phi\circ\epsilon'_d(\eta)\circ\eta\epsilon_d.
\]
Their ratio is an invertible natural endomorphism of $\Phi$. The degree-two and degree-one sector formulas with $(p,q)=(0,0)$ give
\[
\HH^0(\AA_S)=\C.
\]
Conjugation identifies the degree-zero endomorphisms of $\Phi$ with this space. Thus the ratio is a scalar $c\in\C^\times$. Replacing $\eta$ by $c^{-1/2}\eta$ makes the two maps equal. This proves the coherence relation for a $\Z/2$-equivariant functor. The two square roots differ by the nontrivial character. The same sector formulas give $\HH^j(\AA_S)=0$ for $j<0$, so there are no additional higher homotopies obstructing this construction for dg quasi-functors. Equivalently one may use the equivariantization criterion in \cite[Lemma 6.2]{DJR25}, followed by the enhanced construction of \cite{Ela14,Shi18}. We obtain
\begin{equation}\label{eq:equivfunctor}
\Phi^G:\AA_S^G\xrightarrow{\sim}(\AA_{S'})^G.
\end{equation}
This equivalence exists after the scalar choice; it is not asserted to be uniquely linearized.

\begin{theorem}\label{thm:equiv}
Let $S,S'$ be smooth complex del Pezzo surfaces of the same degree $d\in\{1,2\}$, and let $\Phi:\AA_S\simeq\AA_{S'}$ be a $\C$-linear enhanced equivalence. A coherent equivariant lift \eqref{eq:equivfunctor} induces an isomorphism of polarized integral Hodge structures
\[H^1(C,\Z)\cong H^1(C',\Z)\]
for the branch curves. Hence $C\cong C'$. This curve isomorphism extends to the ambient projective plane when $d=2$, and to the weighted projective plane when $d=1$; in both cases it lifts to an isomorphism $S\cong S'$.
\end{theorem}
\begin{proof}
Write $\pi:S\to Y$ for the double cover. Equations \eqref{eq:rootSOD} and \eqref{eq:eqSOD} give two semiorthogonal decompositions of $\Db([S/G])$. The first has the branch category $\Db(C)$ and the base category $\Db(Y)$; the second has $\AA_S^G$ and two exceptional objects. For $Y=\PP^2$, the full exceptional collection has length three \cite{Bei78}; for the smooth stack $\PP(1,1,2)$, the collection $\OO_Y,\OO_Y(1),\OO_Y(2),\OO_Y(3)$ is full exceptional \cite[Remark 2.2.6]{Can06}.

The cyclic-cover comparison of \cite[Theorem 3.4]{DJR25} makes the residual decompositions explicit. Its parameters are $(m,b,M)=(3,2,1)$ and $(4,3,1)$, where $m$ is the Lefschetz length, $b$ is half the branch degree, and $M=m-b$. Thus $0<M<b$, and there are $b-M$ residual exceptional objects:
\begin{equation}\label{eq:equivd2}
\AA_S^G=\langle\Db(C),E\rangle\qquad(d=2),
\end{equation}
\begin{equation}\label{eq:equivd1}
\AA_S^G=\langle\Db(C),E_1,E_2\rangle\qquad(d=1).
\end{equation}
Here the displayed branch embeddings are those obtained by the comparison mutations. The proof uses the root-stack decomposition and the exceptional collection, which also apply to the smooth weighted-projective base. In particular, the length $m$ is not used as a substitute for the dimension of that stack.

There is also a direct way to obtain the precise invariant needed here without comparing the residual embeddings. By additivity of topological K-theory \cite[Theorem 4.1]{DJR25}, the two fully faithful inclusions
\[
\Db(C)\hookrightarrow\Db([S/G]),\qquad
\AA_S^G\hookrightarrow\Db([S/G])
\]
are both isomorphisms on $\Ktop_{-1}$: their exceptional complements have zero odd topological K-groups. Each inclusion preserves the Euler pairing. Consequently their composite inverse gives a Hodge isometry
\begin{equation}\label{eq:Kcurve}
\Ktop_{-1}(\AA_S^G)\cong\Ktop_{-1}(C)\cong H^1(C,\Z).
\end{equation}
This explains why no intrinsic identification of a branch subcategory inside $\AA_S^G$ is required. The same conclusion follows from \eqref{eq:equivd2}--\eqref{eq:equivd1}.

The last isomorphism in \eqref{eq:Kcurve} is the integral edge isomorphism $K^1(C)\cong H^1(C,\Z)$ in the Atiyah--Hirzebruch spectral sequence. For a smooth complex curve, the only odd cohomology is in degree one, so there are neither differentials nor extension ambiguities \cite[Chapter 2]{Ati67}. By \cite[Propositions 4.2 and 4.4]{DJR25}, the enhanced equivalence \eqref{eq:equivfunctor} preserves the Hodge structure and topological Euler pairing. Topological Riemann--Roch identifies the latter on the curve with its integral cup-product pairing, using one fixed Bott sign convention for both curves. We therefore obtain a Hodge isometry
\[
(H^1(C,\Z),\smile)\xrightarrow{\sim}(H^1(C',\Z),\smile),
\]
and an isomorphism of principally polarized Jacobians. By the classical Torelli theorem \cite[Chapter VI]{ACGH85}, there is an isomorphism $\varphi:C\to C'$. It remains to recover the ambient double cover.

\smallskip\noindent\emph{Degree two.}
Adjunction for the smooth plane quartic gives
\begin{equation}\label{eq:canonicalquartic}
K_C\cong\OO_C(1).
\end{equation}
Thus its plane embedding is the canonical embedding. Pullback of differentials is an isomorphism
\[\varphi^*:H^0(C',K_{C'})\xrightarrow{\sim}H^0(C,K_C).\]
Projectivizing its dual gives a projective linear map between the original planes carrying $C$ to $C'$; see \cite[Chapter III, Section 2]{ACGH85}. It identifies the line bundles $\OO_{\PP^2}(2)$ defining the covers. The two quartic sections differ by a nonzero scalar after this identification, and that scalar has a square root over $\C$. Rescaling the trace-zero summand in
\[\Spec_{\PP^2}(\OO_{\PP^2}\oplus\OO_{\PP^2}(-2))\]
lifts the plane isomorphism to $S\cong S'$ \cite[Section 2]{Par91}.

\smallskip\noindent\emph{Degree one.}
Adjunction on $Y=\PP(1,1,2)$ gives
\begin{equation}\label{eq:canonicalsextic}
K_C\cong\OO_C(6-1-1-2)=\OO_C(2).
\end{equation}
The second Veronese map is
\begin{equation}\label{eq:veronese}
\nu_2:Y\longrightarrow\PP^3,\qquad
[x_0:x_1:y]\longmapsto[x_0^2:x_0x_1:x_1^2:y].
\end{equation}
It identifies the coarse weighted plane with the quadric cone $Q=\{u_0u_2-u_1^2=0\}$. Since $C$ avoids the vertex, its image under \eqref{eq:veronese} is a smooth complete intersection of $Q$ and a cubic. The four weight-two sections restrict to all canonical sections, by \eqref{eq:canonicalsextic}; this is a genus-four canonical embedding. Its ideal has a unique quadric generator, so $Q$ is the unique containing quadric \cite[Chapter III, Section 3]{ACGH85}.

The isomorphism on canonical differentials gives a projective isomorphism $h:\PP^3\to\PP^3$ carrying the two canonical curves to one another, and hence carrying $Q$ to $Q'$. To describe its lift explicitly, note that it fixes the vertex and acts on the quotient conic in the coordinates $u_0,u_1,u_2$. Over $\C$ this action comes from an invertible linear change of $x_0,x_1$; the remaining coordinate transforms as
\[y\longmapsto a y+q_2(x_0,x_1),\qquad a\neq0.\]
It therefore lifts to a weight-preserving automorphism of $\PP(1,1,2)$ preserving $\OO_Y(1)$, and in particular $\OO_Y(3)$. The transformed sextic sections have the same zero divisor and differ by a nonzero scalar. Rescaling the trace-zero summand in $\Spec_Y(\OO_Y\oplus\OO_Y(-3))$ gives the required isomorphism of the smooth surface double covers. This completes the proof.
\end{proof}

\section{The smaller Clifford component and spinor modifications}
We now consider a del Pezzo surface $S$ of degree $d\leq4$ as a conic bundle over $\PP^1$, and ask what the corresponding smaller Clifford component remembers. Let $W$ be a two-dimensional vector space. We use $\PP(W)$ for the space of lines in $W$.

\subsection{A common blow-up model}
Hereafter we use the standard construction of conic bundles by blowing up points on distinct fibers of a ruled surface \cite[Section 3.4]{DI09}, taking the ruled surface to be $\PP^1\times\PP^1$.
Put $B=\PP^1$ and $Y=B\times\PP(W)$. Choose distinct $t_i\in B$ and directions $u_i\in\PP(W)$, for $1\leq i\leq r=8-d$, in the locus where
\begin{equation}\label{eq:blowup}
\pi:S_u=\Bl_{(t_1,u_1),\ldots,(t_r,u_r)}Y\longrightarrow Y
\end{equation}
is del Pezzo. Let $f=p_B\circ\pi$, set
\[F=\pi^*p_B^*\OO_B(1),\qquad H=\pi^*p_W^*\OO_{\PP(W)}(1),\]
and denote the exceptional curves by $E_i$. We have
\begin{equation}\label{eq:intersections}
F^2=H^2=0,\quad F\cdot H=1,\quad E_i^2=-1,\quad
F\cdot E_i=H\cdot E_i=E_i\cdot E_j=0\ (i\neq j).
\end{equation}
Consequently
\begin{equation}\label{eq:canonicalblowup}
K_{S_u}=-2F-2H+\sum_iE_i,\quad
K_{S_u/B}=-2H+\sum_iE_i,\quad K_{S_u}^2=8-r=d.
\end{equation}
The fiber over $t_i$ is a transverse union of two lines, and all other fibers are smooth conics. Its discriminant is $D=t_1+\cdots+t_r$. The relatively anticanonical class $2H-\sum_iE_i$ has degree one on each component of a singular fiber and degree two on a smooth fiber. It gives the standard relative plane-conic embedding.

We justify that this del Pezzo locus is nonempty for every prescribed collection of distinct $t_i$. The following elementary argument permits us to choose the centers successively on these fibers. Let $T$ be a del Pezzo surface of degree $e\geq2$, and let $p\in T$. Suppose the strict transform of an integral curve $C$ has nonpositive anticanonical degree after blowing up $p$. Put $k=-K_T\cdot C$ and $m=\operatorname{mult}_pC$. Then $m\geq k$. The genus bound and the Hodge index theorem give
\[
m(m-1)\leq C^2-k+2\leq\frac{k^2}{e}-k+2,
\qquad k^2\left(1-\frac1e\right)\leq2.
\]
If $e\geq3$, this forces $k=1$, and adjunction then gives $C^2=-1$ and $m=1$. For $e=2$ there is one further possibility: $k=m=2$ and $C^2=2$. Equality in the Hodge index theorem gives $C\sim-K_T$; numerical and linear equivalence agree here because $T$ is a rational surface. Such an anticanonical curve is singular at $p$ precisely when $p$ lies on the ramification curve of the degree-two anticanonical map. Thus the blow-up is del Pezzo if $p$ avoids the finitely many $(-1)$-curves, and, in degree two, the ramification curve. Positivity on every integral curve and $K^2=e-1>0$ imply ampleness by the Nakai--Moishezon criterion.

Begin with $T=Y$, and at the $i$-th step choose $p$ on the smooth fiber over $t_i$. This fiber is not a $(-1)$-curve. In degree two it is also not the ramification curve, which is an irreducible curve of genus three. The forbidden set on this fiber is therefore finite, so a choice is always possible. This proves nonemptiness. Ampleness is open in the family of these blow-ups, so the del Pezzo locus is a nonempty open subset of $(\PP(W))^r$.

The ordered data $(t,u)$ modulo $\PGL_2\times\PGL_2$ have dimension $2r-6=10-2d$. Fixing the first coordinates leaves $r-3=5-d$ parameters modulo $\PGL(W)$. A fixed del Pezzo surface has only finitely many conic-bundle classes, since these are integral classes $h$ satisfying $h^2=0$ and $-K_S\cdot h=2$ in a lattice with negative-definite orthogonal complement to $K_S$. For each fibration there are only finitely many choices of components of its singular fibers to contract. A resulting identification with the trivial ruled surface is unique up to $\PGL(W)$ over the fixed base. Hence the map from our quotient parameter space to isomorphism classes of surfaces has finite fibers. In particular its image has dimension $5-d$, and two general members are non-isomorphic.

\subsection{The explicit abstract spinor}
Choose nonzero vectors $\widetilde v_i\in W$, and write $v_i=[\widetilde v_i]$. The divisor sequence
\[0\longrightarrow\OO_{S_u}\longrightarrow\OO_{S_u}(E_i)
\longrightarrow\OO_{E_i}(-1)\longrightarrow0\]
identifies $\Ext^1(\OO_{E_i}(-1),\OO_{S_u}\otimes W)$ with $W$. Use the class $\widetilde v_i$ in each summand and form
\begin{equation}\label{eq:spinorext}
0\longrightarrow\OO_{S_u}\otimes W\longrightarrow\mathcal G_v
\longrightarrow\bigoplus_{i=1}^r\OO_{E_i}(-1)\longrightarrow0,
\qquad \mathcal F_v:=\mathcal G_v(-H).
\end{equation}
Scaling $\widetilde v_i$ changes the identification of the quotient summand and leaves the middle bundle unchanged up to isomorphism. Near $E_i$, choose a complement $w_i$ to $\widetilde v_i$. Then
\[\mathcal G_v\cong\OO(E_i)\widetilde v_i\oplus\OO w_i.\]
In particular $\mathcal G_v$ and $\mathcal F_v$ are locally free of rank two.
\begin{proposition}\label{prop:abstract}
The bundle $\mathcal F_v$ is an abstract spinor on $S_u/B$. More precisely,
\[
\rk\mathcal F_v=2,\quad c_1(\mathcal F_v)=K_{S_u/B},\quad
c_2(\mathcal F_v)=0,\quad Rf_*\mathcal F_v=0.
\]
\end{proposition}
\begin{proof}
From the divisor sequence,
\[\ch(\OO_{E_i}(-1))=E_i+\tfrac12E_i^2.\]
Thus \eqref{eq:spinorext} gives $c_1(\mathcal G_v)=\sum_iE_i$ and $c_2(\mathcal G_v)=0$. Twisting a rank-two bundle by $\OO(-H)$ and using \eqref{eq:intersections} gives
\[
c_1(\mathcal F_v)=\sum_iE_i-2H=K_{S_u/B},\qquad c_2(\mathcal F_v)=0.
\]
Furthermore
\[
Rf_*\OO_{S_u}(-H)=Rp_{B*}\OO_{B\times\PP(W)}(0,-1)=0,
\]
and
\[
Rf_*\OO_{E_i}(-1)=R\Gamma(\PP^1,\OO(-1))\otimes\OO_{t_i}=0.
\]
The restriction of $H$ to $E_i$ is trivial. Twisting \eqref{eq:spinorext} by $-H$ and applying $Rf_*$ proves the final vanishing.
\end{proof}

\subsection{The endomorphism order and the modified surface}
Set
\begin{equation}\label{eq:flagorder}
\mathcal R_v=\{a\in\End(W)\otimes\OO_B\mid
 a(t_i)(\C\widetilde v_i)\subset\C\widetilde v_i\text{ for every }i\}.
\end{equation}
Locally at $t_i$, in a basis adapted to $\widetilde v_i$, this is the order \eqref{eq:Iwahori}. The local splitting of $\mathcal G_v$ gives
\[
\cEnd(\mathcal G_v)\cong
\begin{pmatrix}\OO&\OO(E_i)\\\OO(-E_i)&\OO\end{pmatrix}.
\]
The direct images of $\OO$ and $\OO(E_i)$ are $\OO_B$, with no higher direct image. For $\OO(-E_i)$, the blow-down gives the ideal of $(t_i,u_i)$, whose direct image by $p_B$ is $\OO_B(-t_i)$, again with no higher direct image. Indeed the evaluation $\OO_B\to\OO_{t_i}$ is surjective. These identifications are compatible with multiplication in the generic matrix algebra $\End(W)\otimes\C(B)$. Consequently
\begin{equation}\label{eq:endorder}
Rf_*\cEnd(\mathcal F_v)=f_*\cEnd(\mathcal F_v)\cong\mathcal R_v.
\end{equation}
The formula is independent of the source directions $u_i$.

To identify the partner, define
\[S_v=\Bl_{(t_1,v_1),\ldots,(t_r,v_r)}(B\times\PP(W)).\]
Apply the same elementary transformation on $S_v$ using its own directions $v_i$, and denote that bundle by $\mathcal F_{v,v}$. We claim that it is the normalized canonical spinor $\mathcal F^0_{S_v/B}$.

Choose a nonzero volume form on $W$. The Euler sequence on the second factor gives a surjection
\[\OO(-H)\otimes W\longrightarrow\OO\]
whose kernel at a point $[v]\in\PP(W)$ is the line $\C v$ in $W$, after the displayed twist. At the center $(t_i,v_i)$ choose fiber coordinate $s$ with $v_i$ at $s=0$. In a suitable basis the surjection is $(s,1)$. On the blow-up of the ideal $(t-t_i,s)$, both functions in that ideal are divisible by a local equation $e$ of the exceptional curve. The positive modification in the $\widetilde v_i$ direction replaces that basis vector by $e^{-1}\widetilde v_i$, so the surjection extends locally as $(s/e,1)$. It remains surjective. Globally it gives
\[
0\longrightarrow\omega_{S_v/B}\longrightarrow\mathcal F_{v,v}
\longrightarrow\OO_{S_v}\longrightarrow0,
\]
where the kernel follows from the determinant computed in Proposition~\ref{prop:abstract}. Since $Rf_{v*}\mathcal F_{v,v}=0$, its connecting morphism is an isomorphism. After normalizing this scalar, uniqueness in \eqref{eq:canonicalintro} gives
\[\mathcal F_{v,v}\cong\mathcal F^0_{S_v/B}.\]
The normalized Clifford presentation therefore has
\[
\Cl_0(S_v/B)\cong f_{v*}\cEnd(\mathcal F^0_{S_v/B})
\cong\mathcal R_v.
\]
Together with \eqref{eq:endorder}, the reconstruction of a ternary form from its pointwise Clifford algebra in \cite[Proposition 2.7]{Kuz25} and Theorem~\ref{thm:kuzspinor} now give
\begin{equation}\label{eq:partner}
(S_u/B)_{\mathcal F_v}\cong S_v/B.
\end{equation}
This argument compares the normalized algebras themselves, and therefore determines the partner rather than only its Morita class.

\begin{theorem}\label{thm:partners}
Let $1\leq d\leq4$, $r=8-d$, and fix a reduced divisor $D\subset\PP^1$ of degree $r$. On the nonempty del Pezzo locus above:
\begin{enumerate}
\item all surfaces $S_u$ have equivalent $B$-linear Clifford components;
\item every direction tuple $v$ gives the abstract spinor $\mathcal F_v$ of \eqref{eq:spinorext}, whose spinor modification is $S_v$;
\item after quotienting the directions by $\PGL_2$, the fixed-discriminant family has dimension $5-d$ and has generically finite fibers over the moduli of surfaces. In particular $\mathcal K_{f_u}$ alone does not reconstruct $S_u$.
\end{enumerate}
The equivalences in (1) can be chosen t-exact for the t-structures induced from the surfaces.
\end{theorem}
\begin{proof}
Nonemptiness and the moduli dimension were proved in Section 5.1. Proposition~\ref{prop:abstract} constructs the abstract spinor, and \eqref{eq:endorder}--\eqref{eq:partner} identify its modification with $S_v$. The equivalence and its t-exactness follow from Theorem~\ref{thm:kuzspinor}. This also proves Theorem~\ref{thm:spinorintro}.
\end{proof}

\subsection{Canonical and distinguished spinors}
Recall the normalized extension
\begin{equation}\label{eq:canonicalspinor}
0\longrightarrow\omega_{S/B}\longrightarrow\mathcal F^0_{S/B}
\longrightarrow\OO_S\longrightarrow0,
\qquad Rf_*\mathcal F^0_{S/B}=0,
\end{equation}
whose connecting morphism is the identity. Let $i:\Kf\hookrightarrow\AA_S=\langle\Kf,E_f\rangle$, with $E_f=f^*\OO_B(-1)$.
\begin{proposition}\label{prop:adjoint}
The right adjoint of $i$ satisfies
\[i^!E_f\cong\mathcal F^0_{S/B}\otimes f^*\OO_B(-1).\]
\end{proposition}
\begin{proof}
Put $L=\OO_B(-1)$ and tensor \eqref{eq:canonicalspinor} by $f^*L$. For $G\in\Kf$, relative duality gives
\[
\RHom_S(G,\omega_{S/B}\otimes f^*L[1])
\cong\RHom_B(Rf_*G,L)=0.
\]
The extension triangle therefore induces a natural isomorphism
\[
\RHom_S(G,\mathcal F^0_{S/B}\otimes f^*L)
\cong\RHom_S(G,f^*L).
\]
The object on the left belongs to $\Kf$. It represents the functor defining $i^!E_f$, proving the assertion.
\end{proof}
The proposition produces the normalized canonical spinor of the source. Let
$\Phi_{v,u}:\mathcal K_{f_v}\xrightarrow{\sim}\mathcal K_{f_u}$ be the equivalence defined by $\mathcal F_v$ in Theorem~\ref{thm:kuzspinor}, with the normalization fixed above. It sends $\mathcal F^0_{S_v/B}$ to $\mathcal F_v$, and $B$-linearity gives
\begin{equation}\label{eq:transportcone}
\Phi_{v,u}(i_v^!E_{f_v})\cong\mathcal F_v\otimes f_u^*\OO_B(-1).
\end{equation}
Thus the modifying spinor is the transported adjoint image of the exceptional complement on the partner, after undoing the base twist.

\section{Hyperbolic equivalence and the quartic conic-bundle classes}
\subsection{Hyperbolic equivalence over the projective line}
For a flat quadric fibration of relative dimension $n$, hyperbolic reduction along a regular isotropic line subbundle produces a quadric fibration of relative dimension $n-2$. Geometrically, this operation is described fiberwise by projecting from the corresponding isotropic point in its tangent hyperplane section. Conversely, with suitable additional data, hyperbolic extension produces a quadric fibration of higher relative dimension. Hyperbolic reduction and hyperbolic extension generate an equivalence relation called hyperbolic equivalence \cite[Definition 2.12]{Kuz24}. This relation preserves the discriminant, the shifted quadratic cokernel sheaf, and the Morita class of the even Clifford algebra \cite[Proposition 1.1]{Kuz24}. Its relation to spinor modification is the subject of \cite[Conjecture 1.4]{Kuz25}. In this section, we prove the conjecture for simply degenerating conic bundles over the projective line over an algebraically closed field of characteristic different from two. Over projective spaces, the criterion of \cite[Theorem 1.3]{Kuz24} includes additional Witt conditions, which we verify explicitly in the following situation.
\begin{theorem}\label{thm:hyperbolic}
Let $k$ be an algebraically closed field of characteristic different from two. Let $f_i:X_i\to\PP^1_k$, $i=1,2$, be flat conic bundles with generically smooth fibers and reduced discriminant divisors $\Delta_i$. The following conditions are equivalent:
\begin{enumerate}
\item $\Delta_1=\Delta_2$ as divisors on the fixed base $\PP^1_k$;
\item $X_1/\PP^1$ and $X_2/\PP^1$ are hyperbolically equivalent;
\item $X_1/\PP^1$ and $X_2/\PP^1$ are spinor modifications of one another.
\end{enumerate}
In particular, \cite[Conjecture 1.4]{Kuz25} holds for these simply degenerating conic bundles.
\end{theorem}
\begin{proof}
Present the conic bundles by forms $q_i:L_i\to\Sym^2E_i^\vee$, with $\rk E_i=3$, and let
\[\widetilde q_i:E_i\otimes L_i\longrightarrow E_i^\vee,\qquad C_i=\coker(\widetilde q_i).\]
Taking determinants gives
\begin{equation}\label{eq:detdiscr}
\OO_{\PP^1}(\Delta_i)\cong(\det E_i^\vee)^{\otimes2}\otimes L_i^{-3}.
\end{equation}
In particular $\deg L_i\equiv\deg\Delta_i\pmod2$. Twisting the presentation by
\[(E_i,L_i)\longmapsto(E_i\otimes M_i,L_i\otimes M_i^{-2})\]
does not change the conic bundle. If $\Delta_1=\Delta_2=\Delta$, the two $L_i$ have the same parity of degree. Since $\Pic(\PP^1)=\Z$, choose the twists so that $L_1\cong L_2\cong\OO_{\PP^1}(-m)$ for one integer $m$.

At $p\in\Delta$, trivialize the bundles over the completed local ring $k[[t]]$. The symmetric matrices representing $\widetilde q_i$ have determinant $t$ times a unit. Their reductions have rank two: rank at most one would force the determinant to have order at least two. Splitting off the nondegenerate rank-two part and completing squares gives diagonal form $\operatorname{diag}(u_1,u_2,a)$, with $u_1,u_2$ units and $a$ a uniformizer times a unit. All units in $k[[t]]$ have square roots because $k$ is algebraically closed and $2$ is invertible. Thus each form is congruent to $\operatorname{diag}(1,1,t)$.

It follows that $C_i$ is a line bundle on the finite reduced scheme $\Delta$, pushed forward to $\PP^1$. Its self-duality is
\[\overline q_i:C_i\xrightarrow{\sim}\mathcal E xt^1(C_i,\OO_{\PP^1}(-m)).\]
Every line bundle on $\Delta$ is trivial. At each point the induced quadratic form is nonzero on a one-dimensional $k$-space. Rescaling independently by square roots of the ratios gives an isometry
\[(C_1,\overline q_1)\cong(C_2,\overline q_2).\]
The same assertion holds with both cokernels zero if $\Delta$ is empty.

We can now apply \cite[Theorem 1.3]{Kuz24} to the normalized morphisms $E_i(-m)\to E_i^\vee$. When $m$ is even, its first additional condition is equality of the fiberwise Witt classes at a point outside $\Delta$. Over $k$ these nondegenerate forms have the same Witt class because both have rank three. When $m$ is odd, the second condition occurs only if the dimension of the base is divisible by four. It is therefore absent on $\PP^1$. The theorem proves (1)$\Rightarrow$(2).

By \cite[Proposition 1.1(2)]{Kuz24}, hyperbolic equivalence preserves the discriminant, proving (2)$\Rightarrow$(1). By \cite[Corollary 1.3]{Kuz25}, it also implies spinor modification, proving (2)$\Rightarrow$(3). Finally, a spinor modification induces a base-linear Morita equivalence of the even Clifford algebras. For a ternary form with simple degeneration, the algebra is Azumaya exactly away from the discriminant; at a discriminant point it has the local order \eqref{eq:Iwahori}. The Azumaya locus is invariant under base-linear Morita equivalence. Hence the discriminant supports agree, and their reducedness gives equality of divisors. This proves (3)$\Rightarrow$(1).
\end{proof}
\begin{corollary}\label{cor:hyperbolicfamily}
For fixed $D=t_1+\cdots+t_r$ in Theorem~\ref{thm:partners}, all conic bundles $S_u/\PP^1$ in the del Pezzo locus belong to one hyperbolic-equivalence class. The $(r-3)=(5-d)$-dimensional family is therefore contained in one spinor-modification class and one Morita class of even Clifford algebras.
\end{corollary}
\begin{proof}
Every member has reduced discriminant $D$, so the conclusion follows from Theorems~\ref{thm:hyperbolic} and \ref{thm:partners}.
\end{proof}

\subsection{The five pairs of conic-bundle classes}
Let $S$ be a smooth quartic del Pezzo surface, let $P_S\cong\PP^1$ be its anticanonical pencil of quadrics, and let $D_{\mathrm{ac}}=\{\lambda_1,\ldots,\lambda_5\}\subset P_S$ be the divisor of singular members. Let $\mathcal P_S\subset\AA_S$ denote the canonical heart characterized in \cite[Lemma 2.13]{Ela26}.
\begin{proposition}\label{prop:ten}
The ten conic-bundle classes on $S$ are canonically partitioned into five pairs indexed by $D_{\mathrm{ac}}$. For the pair indexed by $\lambda_i$:
\begin{enumerate}
\item under a normalized equivalence
\[\AA_S\simeq\Db\bigl(\coh\rootstack{P_S}{D_{\mathrm{ac}}}\bigr),\]
identifying $\mathcal P_S$ with the coherent heart, the two objects $E_h=f_h^*\OO_{\PP^1}(-1)=\OO_S(-h)$ are the two exceptional simple sheaves over $\lambda_i$;
\item for $h'=-K_S-h$, one has $S_{\AA_S}(E_h)\cong E_{h'}[1]$ and $\mathcal K_{f_h}=E_h^\perp\subset\AA_S$;
\item perpendicular reduction gives an equivalence
\[\mathcal K_{f_h}\simeq\Db\bigl(\coh\rootstack{P_S}{D_{\mathrm{ac}}\setminus\{\lambda_i\}}\bigr).\]
\end{enumerate}
Thus the ten decompositions $\AA_S=\langle\mathcal K_{f_h},E_h\rangle$ arise by choosing a stacky point and one of its two simple objects.
\end{proposition}
\begin{proof}
Write $S=\Bl_{p_1,\ldots,p_5}\PP^2$, with basis $H,E_1,\ldots,E_5$ of its Picard group. Its conic-bundle fiber classes are
\[h_i=H-E_i,\qquad h'_i=2H-\sum_{j\neq i}E_j,\qquad 1\leq i\leq5.\]
They satisfy $h_i^2=(h'_i)^2=0$, $(-K_S)\cdot h_i=(-K_S)\cdot h'_i=2$, and $h_i+h'_i=-K_S$. The singular quadric corresponding to this pair is a corank-one cone over $\PP^1\times\PP^1$. Its vertex does not lie on $S$. Projection from the vertex restricts to a degree-two map from $S$ to the quadric surface, and its two projections give the two conic bundles. This identifies the five pairs with the singular quadrics.

For any fiber class $h$, equation \eqref{eq:conicsod} gives $\AA_S=\langle\mathcal K_{f_h},\OO_S(-h)\rangle$. By \cite[Lemmas 2.13--2.16]{Ela26}, the ten line bundles $\OO_S(-h)$ are precisely the ten exceptional simple objects in the canonical heart. They occur in pairs over the five weighted points, and
\[S_{\AA_S}(\OO_S(-h_i))\cong\OO_S(-h'_i)[1].\]
Semiorthogonality gives $\mathcal K_{f_h}=\OO_S(-h)^\perp$. The perpendicular calculus of \cite[Section 9]{GL91} lowers the supporting weight of an exceptional simple by one. At a weight-two point this deletes the stacky structure, which proves the last assertion.
\end{proof}
The base in (3) is the anticanonical pencil $P_S$. This is an abstract description of the perpendicular category. A conic fibration $f_h:S\to B_h$ also gives its own action of $\Perf(B_h)$ on that category. An identification of the two projective lines, and of their actions, is additional data when a base-linear statement is required. The construction of Section 5 always retains the original fixed conic-bundle base.

\section{Weighted projective lines and canonical algebras}
We translate the preceding geometry into the representation theory of canonical algebras. For $\mathcal X_D=\rootstack{\PP^1}{D}$ and $r=\deg D$, the all-weight-two lattice is
\[
\mathbb L=\langle\vec c,\vec x_1,\ldots,\vec x_r\mid2\vec x_i=\vec c\rangle,
\qquad\vec\omega=(r-2)\vec c-\sum_i\vec x_i.
\]
Normalizing $\deg\vec c=2$ gives
\begin{equation}\label{eq:orbdegree}
\deg\vec\omega=r-4,\qquad\chi_{\mathrm{orb}}(\mathcal X_D)=2-\frac r2.
\end{equation}
Thus $r=4$ is tubular, while every $r\geq5$ occurring here is wild. The canonical tilting bundle becomes
\begin{equation}\label{eq:tilting2}
T_{\mathrm{can}}=\OO\oplus\bigoplus_{i=1}^r\OO(\vec x_i)\oplus\OO(\vec c).
\end{equation}
With our right-module convention, put $\Lambda_D=\End(T_{\mathrm{can}})$. Its quiver has source $0$, sink $\infty$, and an intermediate vertex on each arm,
\[0\xrightarrow{a_i}i\xrightarrow{b_i}\infty,\qquad1\leq i\leq r.\]
After sending $t_1,t_2,t_3$ to $\infty,0,1$ and writing $t_j=\lambda_j$, the relations can be chosen as
\begin{equation}\label{eq:relations}
b_i a_i=b_2a_2-\lambda_i b_1a_1,\qquad i=3,\ldots,r,\qquad\lambda_3=1.
\end{equation}
Thus
\begin{equation}\label{eq:canonicalequiv}
\Kf\simeq\Db(\coh\mathcal X_D)\simeq\Db(\modu\text{-}\Lambda_D).
\end{equation}
\begin{center}\small
\begin{tabular}{ccccc}
\toprule
degree $d$&$r=8-d$&weight type&representation type&vertices\\\midrule
4&4&$(2,2,2,2)$&tubular&6\\
3&5&$(2,2,2,2,2)$&wild&7\\
2&6&$(2^6)$&wild&8\\
1&7&$(2^7)$&wild&9\\\bottomrule
\end{tabular}
\end{center}
For a quartic del Pezzo surface, the conic component $\Kf$ is the four-point tubular category, while the full orthogonal $\AA_S$ is the five-point wild category associated with the anticanonical pencil. For a cubic del Pezzo surface, $\Kf$ is already a five-point wild category; it is equivalent to the full orthogonal of the quartic surface whose pencil has those five discriminant points. It is still strictly smaller than the cubic surface's $\AA_S$.

\subsection{What a spinor equivalence means algebraically}
Fix the coarse map $\pi_D:\mathcal X_D\to B$ and the tilting equivalence
\[\Psi_D=\RHom(T_{\mathrm{can}},-):\Db(\coh\mathcal X_D)\xrightarrow{\sim}\Db(\modu\text{-}\Lambda_D).\]
The $B$-linear structure becomes the action
\begin{equation}\label{eq:baseaction}
\begin{split}
\rho_D:\Perf(B)&\longrightarrow\End^{\mathrm{ex}}\bigl(\Db(\modu\text{-}\Lambda_D)\bigr),\\
M&\longmapsto\Psi_D\bigl(\pi_D^*M\otimes\Psi_D^{-1}(-)\bigr).
\end{split}
\end{equation}
This records the marking of the fixed conic-bundle base. The weighted-line category and its coarse-line action do not specify the direction data $v$ selecting a conic bundle.

There is an explicit vector bundle on $\mathcal X_D$ giving the flagged order $\mathcal R_v$. Let $\mathcal D_i$ denote the root divisor over $t_i$. Inside the meromorphic bundle associated with $\pi_D^*(\OO_B\otimes W)$, allow a simple pole along $\mathcal D_i$ only in the line $\C\widetilde v_i$. Denote the resulting bundle by $P_v$. Locally it is
\[
P_v\cong\OO(\mathcal D_i)\widetilde v_i\oplus\OO w_i.
\]
The distinct root divisors are disjoint, so these modifications define a global rank-two bundle. At each stacky point its two fiber characters are the trivial and nontrivial characters. It is therefore a generator for the coherent categories locally over the coarse base.

Since $\pi_{D*}\OO(\mathcal D_i)=\OO_B$ and $\pi_{D*}\OO(-\mathcal D_i)=\OO_B(-t_i)$, the same matrix calculation as in Section 5 gives
\begin{equation}\label{eq:rootendom}
\mathcal R_v\cong\pi_{D*}\cEnd_{\mathcal X_D}(P_v).
\end{equation}
There are no higher direct images because the coarse morphism is tame and cohomologically affine. Under $\Psi_D$ the bundle becomes the distinguished complex
\[\mathsf P_v:=\RHom_{\mathcal X_D}(T_{\mathrm{can}},P_v)\in\Db(\modu\text{-}\Lambda_D).\]
The complex, together with $\rho_D$, recovers the algebra sheaf in \eqref{eq:rootendom}: for every $M\in\Perf(B)$,
\[
\RHom_B(M,\mathcal R_v)\cong
\RHom_{\Lambda_D}(\rho_D(M)\mathsf P_v,\mathsf P_v).
\]
These isomorphisms are compatible with composition. Thus the relative endomorphism algebra is determined, including its multiplication. The ordinary finite-dimensional algebra $\End(\mathsf P_v)$ supplies only its global sections.

The transported object in \eqref{eq:transportcone} gives another realization of the same relative order. Indeed, for any fixed $B$-linear enhanced equivalence $\Theta:\mathcal K_{f_u}\simeq\Db(\coh\mathcal X_D)$, let $Q_v=\Theta(\mathcal F_v)$. Base-linearity and full faithfulness give
\[R\pi_{D*}R\mathcal H om(Q_v,Q_v)\cong\mathcal R_v.\]
The explicit bundle $P_v$ makes a convenient choice for \eqref{eq:rootendom}; this argument identifies the relative endomorphism algebras and does not require an equality $Q_v\cong P_v$ for an arbitrary chosen $\Theta$.

Algebraically, spinor modification is represented by the replacement
\begin{equation}\label{eq:pointed}
(\Lambda_D,\rho_D,\mathsf P_u)\longmapsto
(\Lambda_D,\rho_D,\mathsf P_v),
\end{equation}
followed by reconstruction of the flagged order using \eqref{eq:rootendom}.
\begin{proposition}\label{prop:representation}
Let $d=4$, so $D$ has four points and $\Lambda_D$ is the six-vertex tubular canonical algebra of type $(2,2,2,2)$. For $u,v$ in the del Pezzo locus, the replacement \eqref{eq:pointed} recovers $\mathcal R_u\mapsto\mathcal R_v$, and the corresponding conic bundles are $S_u/B$ and $S_v/B$. The latter is the spinor modification $(S_u/B)_{\mathcal F_v}$.
\end{proposition}
\begin{proof}
Equation \eqref{eq:canonicalequiv} gives the canonical-algebra model of the Clifford component, and \eqref{eq:baseaction} retains its base action. The relative Hom formula reconstructs the order $\mathcal R_v$ from the distinguished complex. By \eqref{eq:partner}, the normalized conic bundle associated with this algebra is $S_v/B$, and it is the modification defined by $\mathcal F_v$. This proves the assertion.
\end{proof}


\begin{thebibliography}{ACGH85}
\bibitem[Ahm10]{Ahm10} I. Ahmed, Homogeneous polynomials with isomorphic Milnor algebras, \emph{Czechoslovak Math. J.} \textbf{60} (2010), no. 1, 125--131. \href{https://doi.org/10.1007/s10587-010-0003-9}{doi:10.1007/s10587-010-0003-9}.
\bibitem[Ahm12]{Ahm12} I. Ahmed, Weighted homogeneous polynomials with isomorphic Milnor algebras, \emph{J. Prime Res. Math.} \textbf{8} (2012), 106--114. \href{https://arxiv.org/abs/0909.5429}{arXiv:0909.5429}.
\bibitem[ACGH85]{ACGH85} E. Arbarello, M. Cornalba, P. A. Griffiths, and J. Harris, \emph{Geometry of Algebraic Curves}, Volume I, Grundlehren Math. Wiss. 267, Springer, 1985.
\bibitem[AB21]{AB21} K. Ascher and D. Bejleri, Moduli of double covers and degree one del Pezzo surfaces, \emph{Eur. J. Math.} \textbf{7} (2021), 557--569.
\bibitem[Ati67]{Ati67} M. F. Atiyah, \emph{K-theory}, W. A. Benjamin, New York--Amsterdam, 1967.
\bibitem[BFK14]{BFK14} M. Ballard, D. Favero, and L. Katzarkov, A category of kernels for equivariant factorizations and its implications for Hodge theory, \emph{Publ. Math. Inst. Hautes \'Etudes Sci.} \textbf{120} (2014), 1--111.
\bibitem[Bei78]{Bei78} A. A. Beilinson, Coherent sheaves on $\PP^n$ and problems of linear algebra, \emph{Funct. Anal. Appl.} \textbf{12} (1978), 214--216.
\bibitem[BFN10]{BFN10} D. Ben-Zvi, J. Francis, and D. Nadler, Integral transforms and Drinfeld centers in derived algebraic geometry, \emph{J. Amer. Math. Soc.} \textbf{23} (2010), no. 4, 909--966. \href{https://doi.org/10.1090/S0894-0347-10-00669-7}{doi:10.1090/S0894-0347-10-00669-7}.
\bibitem[BS20]{BS20} D. Bergh and O. M. Schn\"urer, Conservative descent for semi-orthogonal decompositions, \emph{Adv. Math.} \textbf{360} (2020), 106882.
\bibitem[BO01]{BO01} A. Bondal and D. Orlov, Reconstruction of a variety from the derived category and groups of autoequivalences, \emph{Compositio Math.} \textbf{125} (2001), 327--344.
\bibitem[Cad07]{Cad07} C. Cadman, Using stacks to impose tangency conditions on curves, \emph{Amer. J. Math.} \textbf{129} (2007), 405--427.
\bibitem[Can06]{Can06} A. Canonaco, \emph{The Beilinson complex and canonical rings of irregular surfaces}, Mem. Amer. Math. Soc. \textbf{183} (2006), no. 862.
\bibitem[CI04]{CI04} D. Chan and C. Ingalls, Non-commutative coordinate rings and stacks, \emph{Proc. London Math. Soc.} \textbf{88} (2004), 63--88.
\bibitem[DJR25]{DJR25} H. Dell, A. Jacovskis, and F. Rota, Cyclic covers: Hodge theory and categorical Torelli theorems, \emph{Int. Math. Res. Not. IMRN} (2025), no. 10, rnaf120. \href{https://arxiv.org/abs/2310.13651}{arXiv:2310.13651}.
\bibitem[Dol82]{Dol82} I. Dolgachev, Weighted projective varieties, in \emph{Group Actions and Vector Fields}, Lecture Notes in Math. 956, Springer, 1982, 34--71.
\bibitem[Ela12]{Ela12} A. Elagin, Descent theory for semiorthogonal decompositions, \emph{Sb. Math.} \textbf{203} (2012), 645--676.
\bibitem[Ela14]{Ela14} A. Elagin, On equivariant triangulated categories, \href{https://arxiv.org/abs/1403.7027}{arXiv:1403.7027} (2014).
\bibitem[Ela26]{Ela26} A. Elagin, A categorical Torelli theorem for quartic del Pezzo surfaces, \href{https://arxiv.org/abs/2603.26579}{arXiv:2603.26579} (2026).
\bibitem[GL87]{GL87} W. Geigle and H. Lenzing, A class of weighted projective curves arising in representation theory of finite-dimensional algebras, in \emph{Singularities, Representation of Algebras, and Vector Bundles}, Lecture Notes in Math. 1273, Springer, 1987, 265--297.
\bibitem[GL91]{GL91} W. Geigle and H. Lenzing, Perpendicular categories with applications to representations and sheaves, \emph{J. Algebra} \textbf{144} (1991), 273--343.
\bibitem[Kuz08]{Kuz08} A. Kuznetsov, Derived categories of quadric fibrations and intersections of quadrics, \emph{Adv. Math.} \textbf{218} (2008), 1340--1369.
\bibitem[Kuz19]{Kuz19} A. Kuznetsov, Calabi--Yau and fractional Calabi--Yau categories, \emph{J. Reine Angew. Math.} \textbf{753} (2019), 239--267. \href{https://arxiv.org/abs/1509.07657}{arXiv:1509.07657}.
\bibitem[Kuz24]{Kuz24} A. Kuznetsov, Quadric bundles and hyperbolic equivalence, \emph{Geom. Topol.} \textbf{28} (2024), 1287--1339. \href{https://arxiv.org/abs/2108.01546}{arXiv:2108.01546}.
\bibitem[Kuz25]{Kuz25} A. Kuznetsov, Spinor modifications of conic bundles and derived categories of 1-nodal Fano threefolds, \href{https://arxiv.org/abs/2502.02082}{arXiv:2502.02082} (2025).
\bibitem[KP17]{KP17} A. Kuznetsov and A. Perry, Derived categories of cyclic covers and their branch divisors, \emph{Selecta Math. (N.S.)} \textbf{23} (2017), 389--423.
\bibitem[KP21]{KP21} A. Kuznetsov and A. Perry, Serre functors and dimensions of residual categories, \href{https://arxiv.org/abs/2109.02026}{arXiv:2109.02026} (2021).
\bibitem[LRZ24]{LRZ24} X. Lin, J. V. Rennemo, and S. Zhang, IVHS via Kuznetsov components and categorical Torelli theorems for weighted hypersurfaces, \href{https://arxiv.org/abs/2408.08266}{arXiv:2408.08266} (2024).
\bibitem[LZ24]{LZ24} X. Lin and S. Zhang, Serre algebra, matrix factorization and categorical Torelli theorem for hypersurfaces, \emph{Math. Ann.} \textbf{391} (2025), no. 1, 163--177; published online in 2024.
\bibitem[LZ26a]{LZ26a} X. Lin and S. Zhang, Reconstruction of smooth Fano varieties which are anticanonically quadratically defined in projective space via orthogonal complement of the structure sheaf in arbitrary dimension, in preparation.
\bibitem[LZ26b]{LZ26b} X. Lin and S. Zhang, Categorical Torelli theorem for complete intersection of quadrics and conic bundles, in preparation.
\bibitem[LZ26c]{LZ26c} X. Lin and S. Zhang, Categorical Torelli theorem for $(2,3)$-complete intersection in $\PP^5$, and beyond, in preparation.
\bibitem[LZ26d]{LZ26d} X. Lin and S. Zhang, Categorical Griffiths residue, Cayley trick for complete intersection of hypersurfaces, in preparation.
\bibitem[MY82]{MY82} J. Mather and S.-T. Yau, Classification of isolated hypersurface singularities by their moduli algebras, \emph{Invent. Math.} \textbf{69} (1982), 243--251.
\bibitem[Orl09]{Orl09} D. Orlov, Derived categories of coherent sheaves and triangulated categories of singularities, in \emph{Algebra, Arithmetic, and Geometry: In Honor of Yu. I. Manin}, Vol. II, Progr. Math. 270, Birkh\"auser, 2009, 503--531.
\bibitem[Par91]{Par91} R. Pardini, Abelian covers of algebraic varieties, \emph{J. Reine Angew. Math.} \textbf{417} (1991), 191--213.
\bibitem[Shi18]{Shi18} E. Shinder, Group actions on categories and Elagin's theorem revisited, \emph{Eur. J. Math.} \textbf{4} (2018), 413--422.
\bibitem[Toe07]{Toe07} B. To\"en, The homotopy theory of dg-categories and derived Morita theory, \emph{Invent. Math.} \textbf{167} (2007), no. 3, 615--667.
\bibitem[RVdB19]{RVdB19} A. Rizzardo and M. Van den Bergh, A note on non-unique enhancements, \emph{Proc. Amer. Math. Soc.} \textbf{147} (2019), no. 2, 451--453.
\bibitem[BLS16]{BLS16} D. Bergh, V. A. Lunts, and O. M. Schn\"urer, Geometricity for derived categories of algebraic stacks, \emph{Selecta Math. (N.S.)} \textbf{22} (2016), no. 4, 2535--2568. \href{https://doi.org/10.1007/s00029-016-0280-8}{doi:10.1007/s00029-016-0280-8}.
\bibitem[DI09]{DI09} I. V. Dolgachev and V. A. Iskovskikh, Finite subgroups of the plane Cremona group, in \emph{Algebra, Arithmetic, and Geometry: In Honor of Yu. I. Manin}, Vol. I, Progr. Math. 269, Birkh\"auser Boston, 2009, 443--548. \href{https://doi.org/10.1007/978-0-8176-4745-2_11}{doi:10.1007/978-0-8176-4745-2\_11}.
\bibitem[KP18]{KP18} A. Kuznetsov and A. Perry, Derived categories of Gushel--Mukai varieties, \emph{Compos. Math.} \textbf{154} (2018), no. 7, 1362--1406. \href{https://doi.org/10.1112/S0010437X18007091}{doi:10.1112/S0010437X18007091}.
\bibitem[BP23]{BP23} A. Bayer and A. Perry, Kuznetsov's Fano threefold conjecture via K3 categories and enhanced group actions, \emph{J. Reine Angew. Math.} \textbf{800} (2023), 107--153. \href{https://doi.org/10.1515/crelle-2023-0021}{doi:10.1515/crelle-2023-0021}.
\bibitem[JLLZ24]{JLLZ24} A. Jacovskis, X. Lin, Z. Liu, and S. Zhang, Categorical Torelli theorems for Gushel--Mukai threefolds, \emph{J. Lond. Math. Soc. (2)} \textbf{109} (2024), no. 3, e12878. \href{https://doi.org/10.1112/jlms.12878}{doi:10.1112/jlms.12878}.
\end{thebibliography}
\end{document}